\documentclass[12pt]{amsart}
\usepackage{amssymb}
\usepackage[cmtip,all]{xy}

\usepackage{pgf, tikz}
\definecolor {processblue}{cmyk}{0.96,0,0,0}
\input xy
\xyoption{all}
\usepackage{mathtools}
\usepackage{amsfonts}
\usepackage{amsthm}

\usepackage{amscd}
\newtheorem{lemma}{Lemma}[section]

\newtheorem{theorem}[lemma]{Theorem}
\newtheorem{proposition}[lemma]{Proposition}

\theoremstyle{definition}
\newtheorem{remark}[lemma]{Remark}
\newtheorem{definition}[lemma]{Definition}

\newcommand{\bsm}{\begin{smallmatrix}}
\newcommand{\esm}{\end{smallmatrix}}
\newcommand{\bbm}{\begin{matrix}}
\newcommand{\ebm}{\end{matrix}}

\begin{document}

\title{A characterization of right 4-Nakayama artin algebras}

\author{Alireza Nasr-Isfahani}
\address{Department of Mathematics\\
University of Isfahan\\
P.O. Box: 81746-73441, Isfahan, Iran\\ and School of Mathematics, Institute for Research in Fundamental Sciences (IPM), P.O. Box: 19395-5746, Tehran, Iran}
\email{nasr$_{-}$a@sci.ui.ac.ir / nasr@ipm.ir}
\author{Mohsen Shekari}
\address{Department of Mathematics\\
University of Isfahan\\
P.O. Box: 81746-73441, Isfahan, Iran}
\email{mshekari@sci.ui.ac.ir}



\begin{abstract}
We characterize right $4$-Nakayama artin algebras which appear naturally in the study of representation-finite artin algebras. For a right $4$-Nakayama artin algebra $\Lambda$, we classify all finitely generated indecomposable right $\Lambda$-modules and then we compute all almost split sequences over $\Lambda$. We also give a characterization of right $4$-Nakayama finite dimensional $K$-algebras in terms of their quivers with relations.
\end{abstract}

\maketitle

\section{introduction}
Let $R$ be a commutative artinian ring. An artin algebra is an $R$-algebra $\Lambda$ that is a finitely generated $R$-module. Let $\Lambda$ be an artin algebra. A right $\Lambda$-module $M$ is called uniserial ($1$-factor serial) if the lattice of its submodules is totally ordered under inclusion. An artin algebra $\Lambda$ is called Nakayama algebra if the indecomposable projective right $\Lambda$-modules as well as the indecomposable projective left $\Lambda$-modules are uniserial. This then implies that all the finitely generated indecomposable right $\Lambda$-modules are uniserial. Nakayama algebras are representation-finite algebras whose module categories completely understood (see \cite{ARS}, \cite{SY} and \cite{Z}). A non-uniserial right $\Lambda$-module $M$ of length $l$ is called $n$-factor serial ($l\geq n>1$), if $\frac{M}{\mathit{rad}^{l-n}(M)}$ is uniserial and $\frac{M}{\mathit{rad}^{l-n+1}(M)}$ is not uniserial (\cite[Definition 2.1]{NS}). An artin algebra $\Lambda$ is called right $n$-Nakayama if every indecomposable right $\Lambda$-module is $i$-factor serial for some $1\leq i\leq n$ and there exists at least one indecomposable $n$-factor serial right $\Lambda$-module \cite[Definition 2.2]{NS}. The authors in \cite{NS} proved that an artin algebra $\Lambda$ is representation-finite if and only if $\Lambda$ is right $n$-Nakayama for some positive integer $n$ \cite[Theorem 2.18]{NS}. Indecomposable modules and almost split sequences over right $2$-Nakayama and right $3$-Nakayama artin algebras are studied in \cite{NS} and \cite{NS1}. In this paper we will study the module categories of right $4$-Nakayama algebras. We first show that an artin algebra $\Lambda$ which is neither Nakayama nor right $2$-Nakayama nor right $3$-Nakayama is right $4$-Nakayama if and only if every indecomposable right $\Lambda$-module of length greater than $5$ is uniserial and every indecomposable right $\Lambda$-module of length $5$ is local. Then we classify all indecomposable modules and almost split sequences over a right $4$-Nakayama artin algebra. We also show that finite dimensional right $4$-Nakayama algebras are special biserial and characterize quivers and relations of finite dimensional right $4$-Nakayama algebras.
The paper is organized as follows. In Section 2 we study $4$-factor serial right $\Lambda$-modules and then we give a characterization of right $4$-Nakayama artin algebras.

In Section 3 we classify all indecomposable modules over right $4$-Nakayama artin algebras, up to isomorphisms.

In section 4 we compute all almost split sequences over right $4$-Nakayama artin algebras.

Finally in the last section we describe the structure of quivers and their relations of finite dimensional right $4$-Nakayama algebras.

\subsection{notation }
Throughout this paper all modules are finitely generated right $\Lambda$-modules and all fields are algebraically closed fields. For a $\Lambda$-module $M$, we denote by $soc(M)$, $top(M)$, $rad(M)$, $\textit{l}(M)$ and $\textit{ll}(M)$  its socle, top, radical, length and Loewy length  of $M$, respectively.
Let $Q=(Q_0, Q_1, s, t)$ be a quiver and $\alpha:i\rightarrow j$ be an arrow in $Q$. One introduces a formal inverse $\alpha^{-1}$ with $s({\alpha}^{-1}) = j$
and $t(\alpha^{-1}) = i$. An edge in $Q$ is an arrow or the inverse of an arrow. To each vertex $i$ in $Q$,
one associates a trivial path, also called trivial walk, $\varepsilon_i$ with $s(\varepsilon_i) = t(\varepsilon_i) = i$. A non-trivial
walk $w$ in $Q$ is a sequence $w=c_1c_2\cdots c_n$, where the $c_i$ is an edge such that $t(c_i) = s(c_{i+1})$ for all $i$, whose inverse $w^{-1}$ is defined to be
the sequence $w^{-1}=c_n^{-1}c_{n-1}^{-1}\cdots c_1^{-1}$. A walk $w$ is called reduced if $c_{i+1}\neq c_i^{-1}$ for each $i$. For $i\in Q_0$, we denote by $i^+$ and $i^-$ the set of arrows starting in $i$ and the set of arrows ending in $i$, respectively.\\

\section{right 4-Nakayama artin algebras}
In this section we first describe right $4$-factor serial modules and then we give a characterization of right $4$-Nakayama artin algebras.

\begin{definition}
\cite[Definitions 2.1 and 2.2]{NS} Let $\Lambda$ be an artin algebra and $M$ be a right $\Lambda$-module of length $l$.
\item[$(1)$] $M$ is called $1$-factor serial (uniserial) if $M$ has a unique composition series.
\item[$(2)$] Let $l\geq n > 1$. $M$ is called $n$-factor serial if $\frac{M}{rad^{l-n}(M)}$ is uniserial and  $\frac{M}{rad^{l-n+1}(M)}$ is not uniserial.
\item[$(3)$] $\Lambda$ is called right $n$-Nakayama if every indecomposable right $\Lambda$-module is $i$-factor serial for some $1\leq i \leq n$ and there exists at least one indecomposable $n$-factor serial right $\Lambda$-module.
\end{definition}

\begin{lemma} \label{l1}
Let $\Lambda$ be an artin algebra and $M$ be a right $\Lambda$-module of length $r$ and Loewy length $t$. The following conditions are equivalent.
\item[$(a)$] $M$ is a $4$-factor serial right $\Lambda$-module.
\item[$(b)$] One of the following conditions hold.
\begin{itemize}
\item[$(i)$]For every $0\leq i \leq r-5$, $rad^{i}(M)$ is local, $rad^{r-4}(M)$ is not local and length of $M$ is either $t+3$ or $t+2$ or $t+1$. Moreover in this case if $r=t+3$, then $rad^{r-4}(M)=soc(M)=S_{1}\oplus S_{2}\oplus S_{3}\oplus S_{4}$, where $S_i$ is a simple submodule of $M$ for each $1\leq i\leq 4$.
\item[$(ii)$] $M$ is not local, $r=4$ and the Loewy length of $M$ is either $2$ or $3$. Moreover in this case if $soc(M)$ is simple, then $ll(M)=3$, otherwise $ll(M)=2$.
\end{itemize}
\end{lemma}
\begin{proof}
$(a)\Longrightarrow (b)$. Assume that  $M$ is a local right $\Lambda$-module. Then by \cite[Theorem 2.6]{NS}, for every $0\leq i \leq r-5$, $rad^{i}(M)$ is local and  $rad^{r-4}(M)$ is not local. On the other hand by \cite[Lemma 2.21]{NS}, $t+1 \leq r \leq t+3$. If $r= t+3$, then by \cite[Remark 2.7]{NS}, $soc(M)\subseteq rad^{r-4}(M) $ and $t=r-3$. Therefore $soc(M)= rad^{r-4}(M) =S_{1}\oplus S_{2}\oplus S_{3}\oplus S_{4}$. This finishes the proof of $(i)$. Assume that $M$ is not local. So by \cite[Corollary 2.8]{NS}, $r=4$ and the result follows.\\
$(b)\Longrightarrow (a)$. If $M$ is not local and $r=4$, then by \cite[Corollary 2.8]{NS}, $M$ is a $4$-factor serial right $\Lambda$-module. If $M$ satisfies the condition $(i)$, then $ \frac{M}{rad^{r-4}(M)} $ is uniserial and $ \frac{M}{rad^{r-3}(M)} $ is not uniserial. Therefore $M$ is a $4$-factor serial right $\Lambda$-module
\end{proof}

An Artin algebra $\Lambda$ is of right $n$-th local type if for every indecomposable right $\Lambda$-module $M$, $top^{n}(M)=\frac{M}{rad^{n}(M)}$ is indecomposable \cite{A1}.

\begin{lemma}\label{l2}
Let $\Lambda$ be a right $4$-Nakayama artin algebra. Then $\Lambda$ is of $3$-ed local type.
\end{lemma}
\begin{proof}
Let $M$ be an indecomposable right $\Lambda$-module. If $M$ is local, then $\frac{M}{rad^{3}(M)}$ is indecomposable. If $M$ is not local, then by \cite[Lemma 5.2]{NS}, $M$ is either $3$-factor serial or $4$-factor serial. By \cite[Lemma 2.1]{NS1} and Lemma \ref{l1}, $ll(M)\leq 3$, so $ \frac{M}{rad^{3}(M)}\cong M $ is indecomposable and the result follows.
\end{proof}

Let $\Lambda$ be an artin algebra.  The valued quiver of $\Lambda$ is a quiver with $n$ vertices, where $n$ is the number of the isomorphism classes of simple right $\Lambda$-modules and with at most one arrow from a vertex $i$ to the vertex $j$ and with an ordered pair of positive integers associated with each arrow. This is done by writing an arrow from $i$ to $j$ if $Ext_{\Lambda}^{1}(S_i,S_j)\neq 0$, where $S_i$ and $S_j$ are simple $\Lambda$-modules corresponding to the vertices $i$ and $j$ and assigning to this arrow the pair of integers  $(dim_{End_{{{\Lambda}}}{(S_j)}} Ext_{\Lambda}^{1}(S_i,S_j), dim_{End_{{{\Lambda}}}{(S_i)}^{op}} Ext_{\Lambda}^{1}(S_i,S_j))$\cite{ARS}. Note that by \cite[Proposition III.1.15]{ARS}, $Ext_{\Lambda}^{1}(S_i,S_j)\neq 0$ if and only if $S_j$ is a direct summand of $\frac{rad (P_i)}{rad^{2} (P_i)}$, where $P_i$ is a projective cover of $S_i$.\\

Let $\Lambda$ be an artin algebra and $M$ be a right $\Lambda$-module. $M$ has a waist if there is a nontrivial proper submodule $N$
of $M$ such that every submodule of $M$ contains $N$ or is contained in $N$. In this case, $N$ is called a waist in $M$.

\begin{proposition}\label{p1}
Let $\Lambda$ be a right $4$-Nakayama artin algebra and $M$ be an indecomposable  local right  $\Lambda$-module. Then the  following statements hold.
\item[$(a)$] If $M$ is  $4$-factor serial, then $l(M)=5$.
\item[$(b)$] If $M$ is  $3$-factor serial, then $l(M)=4$.
\item[$(c)$] If $M$ is  $2$-factor serial, then $l(M)=3$.
\end{proposition}
\begin{proof}
$(a)$. By \cite[Lemma 2.3]{NS1}, $M$ is a projective right $\Lambda$-module and by Lemma \ref{l2}, $\Lambda$ is of $3$-ed local type. If $rad^{3}(M)\neq 0$, then by \cite[Theorem 2.5]{A2}, $rad^{3}(M)$ is uniserial and waist of $M$. Also by \cite[Theorem 2.6]{NS}, $rad(M) $ is $4$-factor serial. Assume that $rad(M)$ is local. Then by \cite[Theorem 2.6]{NS}, $rad^2(M)$ is $4$-factor serial. If $rad^2(M)$ is local, then by \cite[Theorem 2.6]{NS}, $rad^{3}(M)$ is $4$-factor serial which gives a contradiction. Therefore $rad^2(M)$ is non-local and by \cite[Corollary 2.8]{NS} $\textit{l}(rad^2(M))=4$. This implies that $\textit{l}(M)=6$. Let $top(M)=S_1$, $top(rad(M))=S_2$ and $S_3, S_4$ be direct summands of $top(rad^2(M))$. Since $M$ is projective, by \cite[Proposition III.1.15]{ARS}, there exists one arrow from the vertex $1$ to the vertex $2$ in the valued quiver of $\Lambda$. Since $rad(M)$ is local, $rad(M)$ is either projective or quotient of an indecomposable projective right $\Lambda$-module $P_2$. Assume that $rad(M)$ is projective. Since $S_3$ and $S_4$ are direct summands of $top(rad^2(M))$, by \cite[Proposition III.1.15]{ARS}, there are one arrow from the vertex $2$ to the vertex $3$ and one arrow from the vertex $2$ to the vertex $4$ in the valued quiver of $\Lambda$. Then the valued quiver of $\Lambda$ has a subquiver of the form

$$\hskip .5cm 
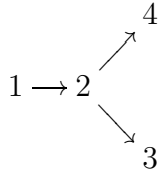
\xymatrix@-4mm{
&&{4}\ar @{<-}[dl]\\
{1}\ar [r]&{2}\ar [dr]\\
&& {3}}\hskip .5cm$$
which implies that there exists a non-local indecomposable right $\Lambda$-module $N$ of length $5$. Therefore by \cite[Corollary 2.8]{NS}, $N$ is a $5$-factor serial right $\Lambda$-module which is a contradiction. Now assume that $rad(M)$ is not projective. Then $rad(M)\cong \frac{P_2}{K}$, where $P_2$ is the indecomposable projective right $\Lambda$-module and $K$ is a submodule of $P_2$ and so $rad^{2}(M)\cong \frac{rad(P_2)}{K}$. Then $S_3$ and $S_4$ are direct summands of $rad(P_2)$ and by \cite[Proposition III.1.15]{ARS}, there are one arrow from the vertex $2$ to the vertex $3$ and one arrow from the vertex $2$ to the vertex $4$ in the valued quiver of $\Lambda$. Then the valued quiver of $\Lambda$ has a subquiver of the form
$$\hskip .5cm 
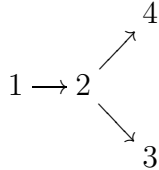
\xymatrix@-4mm{
&&{4}\ar @{<-}[dl]\\
{1}\ar [r]&{2}\ar [dr]\\
&& {3}}\hskip .5cm$$
which implies that there exists a non-local indecomposable right $\Lambda$-module  $N$ of length $5$. Then by \cite[Corollary 2.8]{NS}, $N$ is a $5$-factor serial right $\Lambda$-module which is a contradiction. Therefore $rad(M)$ is non-local and by \cite[Corollary 2.8]{NS}, $rad(M)$ is of the length $4$. It implies that $l(M)=5$.\\
In the parts $(b)$ and $(c)$, $M$ is local. Then $M$ is either projective or quotient of an indecomposable projective. Then by the similar argument as in the proof of the part $(a)$, we can prove parts $(b)$ and $(c)$.
\end{proof}

Now we give a characterization of right $4$-Nakayama artin algebras.

\begin{theorem}\label{t1}
Let $\Lambda$ be an artin algebra which is  neither Nakayama, nor right $2$-Nakayama, nor right $3$-Nakayama.
Then  $\Lambda$ is right $4$-Nakayama if and only if every indecomposable right $\Lambda$-module $M$ of length greater than $5$ is uniserial and every indecomposable right $\Lambda$-module $M$ of length $5$ is local.
\end{theorem}
\begin{proof}
  Let $\Lambda$ be a right $4$-Nakayama algebra. It follows from Proposition \ref{p1} that, every indecomposable right $\Lambda$-module $M$ of length greater than $5$ is uniserial. If there exists an indecomposable non-local  right  $\Lambda$-module $M$ of length $5$, then by  \cite[Corollary 2.8]{NS}, $M$ is $5$-factor serial which gives a contradiction. Conversely, assume that any indecomposable right $\Lambda$-module of length greater than $5$ is uniserial and every indecomposable right $\Lambda$-module of length $5$ is local, so by  \cite[Corollary 2.8]{NS} and  \cite[Lemma 2.21]{NS}, every indecomposable right $\Lambda$-module is $t$-factor serial for some $t\leq 4$. Since $\Lambda$ is neither Nakayama, nor right $2$-Nakayama, nor right $3$-Nakayama, there exists an indecomposable $t$-factor serial right $\Lambda$-module $M$ for some $t\geq 4$. Therefore $\Lambda$ is right $4$-Nakayama.
\end{proof}

\section{indecomposable modules over right $4$-Nakayama artin algebras}

In this section we give a classification of finitely generated indecomposable modules over right $4$-Nakayama artin algebras.

\begin{proposition}\label{p2}
Let $\Lambda$ be a right $4$-Nakayama artin algebra and $M$ be an indecomposable right $\Lambda$-module. Then $l(soc(M))\leq 2$.
\end{proposition}
\begin{proof}
If $M$ is uniserial, then $l(soc(M))= 1$ and if $M$ is $2$-factor serial, then by \cite[Lemma 5.1]{NS}, $l(soc(M))= 2$. Assume that $M$ is $3$-factor serial. If $M$ is non-local, then $l(M)=3$ and by  \cite[Lemma 2.1]{NS1}, $l(soc(M))=1$ and if $M$ is local, then by Proposition \ref{p1}, $l(M)=4$ and $l(soc(M))\leq 3$. Assume that on the contrary that $rad(M)=soc(M)=S_1\oplus S_2\oplus S_3$ and $top(M)=S_4$, where $S_i$ is a simple $\Lambda$-module for each $1\leq i\leq 4$. Then  by the similar argument as in the proof of the Proposition \ref{p1}, the valued quiver of $\Lambda$ has a subquiver of the form
$$\hskip .5cm 
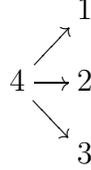
\xymatrix@-4mm{
&{1}\ar @{<-}[dl]\\
{4}\ar [r]\ar[dr]& {2}\\
&{3}}\hskip .5cm$$
Then there exists a non-local indecomposable right $\Lambda$-module  of length $5$ that by  \cite[Corollary 2.8]{NS}  is $5$-factor serial which gives a contradiction. Now assume that $M$ is $4$-factor serial. If $M$  is non-local, then by Lemma \ref{l1}, $l(M)=4$ and since $l(top(M))\geq 2$, $l(soc(M))\leq 2$. If $M$ is local, then by \cite[Lemma 2.2]{NS1}, $M$ is projective and by Proposition \ref{p1}, $l(M)=5$. Assume on the contrary that $l(soc(M))=4$. Then $rad(M)=soc(M)=S_1\oplus S_2\oplus S_3\oplus S_4$, where $S_i$ is a simple $\Lambda$-module for each $1\leq i\leq 4$. Therefore we have almost split sequences
 \begin{center}
$0\longrightarrow S_1\longrightarrow M\longrightarrow \tau^{-1}(S_1)\longrightarrow0 $
\end{center}
\begin{center}
$0\longrightarrow S_2\longrightarrow M\longrightarrow \tau^{-1}(S_2)\longrightarrow0 $
\end{center}
\begin{center}
$0\longrightarrow S_3\longrightarrow M\longrightarrow \tau^{-1}(S_3)\longrightarrow0 $
\end{center}
\begin{center}
$0\longrightarrow S_4\longrightarrow M\longrightarrow \tau^{-1}(S_4)\longrightarrow0 $
\end{center}

$$\xymatrix@-5mm{
{0}\ar[r]&M\ar[r]&{\tau^{-1}(S_1)\oplus \tau^{-1}(S_2)\oplus \tau^{-1}(S_3)\oplus \tau^{-1}(S_4)}\ar[r]&{\tau^{-1}(M)}\ar[r]&{0}}
$$
Then $\tau^{-1}(M)$ is an indecomposable right $\Lambda$-module of length $11$ and by Proposition \ref{p1}, $\tau^{-1}(M)$ is uniserial. Also for each $1\leq i \leq 4$, irreducible morphisms $ \tau^{-1}(S_i)\longrightarrow \tau^{-1}(M) $ are monomorphisms. For each $1\leq i \leq 4$, $\tau^{-1}(S_i)\cong \frac{M}{S_i}$ and by \cite[Theorem 2.13]{NS}, there exists  $1\leq i \leq 4$ such that $\tau^{-1}(S_i)\cong \frac{M}{S_i}$ is not uniserial which gives a contradiction to the Corollary 2.17 of \cite{NS}. Now assume that $l(soc(M))= 3$ and $soc(M)=S_1\oplus S_2\oplus S_3$ where $S_i$ is a simple $\Lambda$-module for each $1\leq i\leq 3$.
If $rad(M)=N\oplus S_2\oplus S_3$, where $N$ is uniserial of length $2$ and $top(N)=S_4$, $soc(N)=S_1$ and $top(M)=S_5$, then by the  similar argument as in the proof of the Proposition \ref{p1}, the valued quiver of $\Lambda$ has a subquiver of the form
$$\hskip .5cm 
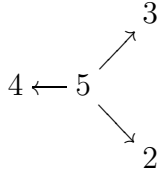
\xymatrix@-4mm{
&&{3}\ar @{<-}[dl]\\
{4}\ar @{<-}[r]&{5}\ar[dr]\\
&&{2}}\hskip .5cm$$
Therefore  there exists  a non-local indecomposable right $\Lambda$-module $L$ of length $5$. By  \cite[Corollary 2.8]{NS}, $L$  is $5$-factor serial which gives a contradiction. If $rad(M)=N \oplus S_3$ where $N$ is a $2$-factor serial right $\Lambda$-module of length $3$, $top(N)=S_4$ and $rad(N)=soc(N)=S_1\oplus S_2$.  The similar argument as in the  proof of  Proposition \ref{p1} shows that the  valued quiver of  $\Lambda$ has a subquiver of the form
$$\hskip .5cm 
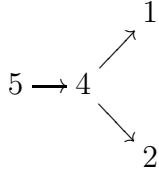
\xymatrix@-4mm{
&&{1}\ar @{<-}[dl]\\
{5}\ar[r]&{4}\ar [dr]\\
&&{2}}\hskip .5cm$$

which gives a contradiction. Therefore $l(soc(M))\leq 2$ and the result follows.
\end{proof}

\begin{proposition}\label{p3}
Let $\Lambda$ be a right $4$-Nakayama artin algebra and $M$ be an indecomposable $4$-factor serial right $\Lambda$-module. Then the following statements hold.
\item[$(a)$] If $M$ is local and $rad(M)$ is indecomposable, then $soc(M)$ is simple.
\item[$(b)$] If $M$ is non-local, then  $l(top(M))=2$.
\end{proposition}
\begin{proof} $(a).$ By \cite[Lemma 2.3]{NS1}, $M$ is projective and by the Proposition \ref{p1}, $l(M)=5$. By the Proposition \ref{p2}, $l(soc(M))\leq 2$ and by \cite[Theorem 2.6, Corollary 2.8]{NS}, $rad(M)$ is  non-local $4$-factor serial of length $4$. Assume that $soc(M)=S_1\oplus S_2$, $top(rad(M))=S_3\oplus S_4$ and $top(M)=S_5$, where $S_i$ is a simple $\Lambda$-module for each $1\leq i\leq 5$. Since $M$ is projective, by \cite[Proposition III.1.15]{ARS} there are one arrow from the vertex $5$ to the vertex $3$ and one arrow from the vertex $5$ to the vertex $4$ in the valued quiver of $\Lambda$. Since $top(rad(M))=S_3\oplus S_4$, $\frac{P_3\oplus P_4}{L}\cong rad(M)$ where $P_i$ is indecomposable projective that $top(P_i)=S_i$ and $L$ is a submodule of $P_3\oplus P_4$. Then  $S_1\oplus S_2=soc(M)=\frac{rad(M)}{rad^2(M)}\subseteq\frac{rad(P_3\oplus P_4)}{rad^2(P_3\oplus P_4)}$. Since $rad(M)$ is indecomposable, $S_1\oplus S_2$ must be a direct summand of either $top(rad(P_3))$ or $top(rad(P_4))$. We can assume that $S_1\oplus S_2$ is a direct summand of $top(rad(P_3))$. By \cite[Proposition III.1.15]{ARS}, there are one arrow from the vertex $3$ to the vertex $2$ and one arrow from the vertex $3$ to the vertex $1$ in the valued quiver of $\Lambda$. Then the valued quiver of $\Lambda$ has a subquiver of the form
$$\hskip .5cm 
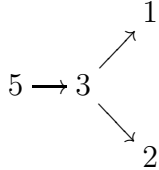
\xymatrix@-4mm{
&&{1}\ar @{<-}[dl]\\
{5}\ar[r]&{3}\ar [dr]\\
&&{2}}\hskip .5cm$$
which implies that there is a non-local indecomposable right $\Lambda$-module $N$ of length $5$. By \cite[Corollary 2.8]{NS}, $N$ is $5$-factor serial which gives a contradiction. Therefore $soc(M)$ is simple. \\
$(b)$. By Lemma \ref{l1}, $l(M)=4$. Assume on the contrary that $top (M)=S_1\oplus S_2\oplus S_3$ and $soc(M)=S_4$, where $S_i$ is a simple $\Lambda$-module for each $1\leq i\leq 4$. $M$ has a projective cover of the form
$f=(f_1, f_2, f_3):P_1\oplus P_2\oplus P_3 \rightarrow M$. Then
 $ \frac{P_1\oplus P_2\oplus P_3}{L}\cong M $, where $L=Ker(f)$ and $S_4=rad(M)\cong  \frac{rad(P_1)\oplus rad(P_2)\oplus rad(P_3)}{L}$. We claim that for each $1\leq i \leq 3 $, $S_4$ is a direct summand of $rad(P_i)$. Assume that $S_4$ is not a direct summand of $rad(P_1)$. Since $S_4=rad(M)$, $rad(P_1)\subset L$. We consider the $\Lambda$-homomorphism $f_1:P_1\rightarrow M$. Then $ \frac{P_1}{rad(P_1)}\cong Im f_1 \leq M$. Since $\frac{P_1}{rad(P_1)}\cong Im f_1$ is simple, $Im f_1$ is direct summand of $soc(M)=S_4$ and $Im f_1=S_4$ which gives a contradiction. The similar argument shows that $S_4$ is a direct summand of $rad(P_2)$ and $rad(P_3)$. So by \cite[Proposition III.1.15]{ARS}, the valued quiver of $\Lambda$ has a subquiver of the form
$$\hskip .5cm 
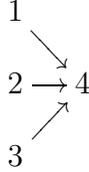
\xymatrix@-4mm{
{1}\ar [dr]\\
{2}\ar[r]&{4}\ar@{<-}[dl]\\
{3}}\hskip .5cm$$
 which gives a contradiction.
\end{proof}

The following theorem gives a classification of submodules of indecomposable modules over right $4$-Nakayama artin algebras.

\begin{theorem}\label{T2}
 Let $\Lambda$  be a right $4$-Nakayama artin algebra and $M$ be an indecomposable right $\Lambda$-module. Then the following statements hold.
 \item[$(a)$] If $M$ is a  $4$-factor serial right  $\Lambda$-module, then  one of the following situations hold:
 \begin{itemize}
\item[$(i)$] $M$ is a local and colocal. Submodules of $M$ are  $rad(M)$ which is indecomposable non-local $4$-factor serial, two indecomposable submodules $M_1$ and $M_2$ of length $3$ that $M_1$ is uniserial and $M_2$ is non-local  $3$-factor serial, two indecomposable uniserial submodules $N_1$ and $N_2$ of length $2$ which $N_1$ is a submodule of $M_1$ and $M_2$, and $N_2$ is a submodule of $M_2$  and $soc(M)=S$ which is simple.
\item[$(ii)$] $M$ is local and non-colocal. Submodules of $M$ are  $rad(M)=N\oplus S_1$ that $N$ is uniserial of length  $3$ and $S_1$ is a simple right $\Lambda$-module, uniserial submodule $N_1$ of  length $2$ which is a submodule of $N$ and $soc(M)=S_1\oplus S_2$ that $soc(N)=S_2$.
\item[$(iii)$] $M$ is local and non-colocal. Submodules of $M$ are $rad(M)=N_1\oplus N_2$ that for each $1\leq i \leq 2$, $N_i$ is uniserial of length $2$ and $soc(M)=S_1\oplus S_2$  where for each $1\leq i \leq 2$,  $soc(N_i)=S_i$.
\item[$(iv)$] $M$ is non-local of length $4$ that $soc(M)$ is simple. Submodules of $M$ are two submodules $M_1$ and $M_2$ of length $3$ that $M_1$ is uniserial and $M_2$ is non-local $3$-factor serial, two indecomposable uniserial submodules $N_1$ and $N_2$ of length $2$ that $N_1$ is a submodule of $M_1$ and $M_2$ and $N_2$ is a submodule of $M_2$ and $soc(M)=S$ which  is simple.
\item[$(v)$] $M$ is non-local of length $4$ that $soc(M)$ is non-simple. Submodules of $M$ are indecomposable modules $N_1$ of length $3$ which is $2$-factor serial, uniserial module $N_2$ of length $2$ and $soc(M)=S_1\oplus S_2$  where $S_i$ is a simple submodule of $M$ for each $1\leq i\leq 2$ and  $soc(N_2)=S_2$.
 \end{itemize}
 \item[$(b)$] If $M$ is a $3$-factor serial right $\Lambda$-module, then one of the following situations hold.
 \begin{itemize}
 \item[$(i)$] $M$ is local and colocal. Submodules of $M$ are $rad(M)$ which is non-local $3$-factor serial of length $3$, two  uniserial submodules $N_1$ and $N_2$ of length $2$ and $soc(M)=S$ which is simple.
 \item[$(ii)$] $M$ is local and $soc(M)$ is not  simple.  Submodules of $M$ are  $rad(M)=N\oplus S_1$ that $N$ is uniserial  of length $2$, $S_1$ is simple and $soc(M)=S_1\oplus S_2$ that $soc(N)=S_2$ which is simple.
 \item[$(iii)$] $M$ is non-local of length $3$. Submodules of $M$ are  two uniserial submodules $N_1$ and $N_2$ of length $2$ and $soc(M)=S$ which  is simple.
 \end{itemize}
 \item[$(c)$] If $M$ is a $2$-factor serial local module, then submodules of $M$ are $rad(M)=soc(M)=S_1\oplus S_2$ that $S_1$ and $S_2$ are simple right $\Lambda$-modules.
 \end{theorem}
\begin{proof}
We prove only parts $(a)(i)$ and $(a)(v)$, the proof of the other parts is similar.
\item[$(a)(i).$] Since $M$ is a local $4$-factor serial right $\Lambda$-module, by the Proposition \ref{p1}, $l(M)=5$ and it implies that $l(rad(M))=4$. By \cite[Theorem 2.6]{NS}, $rad(M)$ is $4$-factor serial and so by \cite[Corollary 2.8]{NS}, $rad(M)$ is non-local indecomposable. By the Proposition \ref{p3}, $l(top(rad(M)))=2$. Assume that $top(rad(M))=S_1\oplus S_2$, where $S_1$ and $S_2$ are simple $\Lambda$-modules. We claim that $rad(M)$ has two maximal submodules. Since $rad(M)$ is non-local, $rad(M)$ has at least two maximal submodules $M_1$ and $M_2$ that $\frac{rad(M)}{M_1}\cong S_2$ and $\frac{rad(M)}{M_2}\cong S_1$. It implies that $S_1$ is a direct summand of $top(M_1)$, $S_2$ is a direct summand of $top(M_2)$ and we have an exact sequence
\begin{center}
 $0\longrightarrow rad^{2}(M)\longrightarrow rad(M)\longrightarrow \frac{ rad(M)}{rad^{2}(M)}\longrightarrow 0$
 \end{center}
Hence $l(rad^2 (M))=2$. $l(rad(M))=4$ and the length of any maximal submodule of $rad(M)$ is $3$, then $rad(M)$ has two maximal submodules $M_1$ and $M_2$. Since for every $1\leq i \leq 2$, $soc(M)=soc(M_i)$ which is simple, $M_i$ is indecomposable and by \cite[Theorem 5.2]{NS}, $M_1$ and $M_2$ are not $2$-factor serial.
If $M_i$ is local, then $M_i$ is uniserial and if $M_i$ is non-local, then $M_i$ is $3$-factor serial for $1\leq i \leq 2$.
 We show that $M_1$ is uniserial and $M_2$ is $3$-factor serial. Assume that $S_3=top(rad^2(M))$ and $S_4=soc(M)$, where $S_3$ and $S_4$ are simple. If $M_1$ and $M_2$ are uniserial, then $top(M_1)=S_1$, $top(M_2)=S_2$ and $P_1, P_2$ are projective covers of $M_1, M_2$, respectively. Since $rad^2(M)\subset M_i$, $S_3$ is a direct summand of $top(rad(P_i))$ for each $1\leq i \leq 2$. By \cite[Proposition III.1.15]{ARS}, there are one arrow from the vertex $1$ to the vertex $3$ and one arrow from the vertex $2$ to the vertex $3$ in the valued quiver of $\Lambda$. Also $P_3$ is the projective cover of $rad^2(M)$ and hence $S_4$ is a direct summand of $top(rad(P_3))$. By \cite[Proposition III.1.15]{ARS}, there is one arrow from the vertex $3$ to the vertex $4$ in the valued quiver of $\Lambda$. Then the valued quiver
  of $\Lambda$ has a subquiver of the form
$$\hskip .5cm \xymatrix@-4mm{
{1}\ar [dr]\\
&{3}\ar@{<-} [dl]\ar[r]&{4}\\
{2}}\hskip .5cm$$
Therefore, there exists a non-local indecomposable right $\Lambda$-module $L$ of length $5$. By \cite[Corollary 2.8]{NS}, $L$ is $5$-factor serial which gives a contradiction. Now assume that $M_1$ and $M_2$ are $3$-factor serial. Since $ \frac{rad(M)}{M_1}\cong S_2 $ and $\frac{rad(M)}{M_2}\cong S_1$, $S_2$ is not a direct summand of  $top(M_1)$ and $S_1$ is not a direct summand of $top(M_2)$. Since $rad(M_i)\subset M_i$ for every $1\leq i \leq 2$, $ top(M_1)=S_1\oplus S_3$ and $top(M_2)=S_2\oplus S_3$ and so $P_1\oplus P_3, P_2\oplus P_3$ are projective covers of $M_1$ and $M_2$, respectively. Then by \cite[Proposition III.1.15]{ARS}, the valued quiver of $\Lambda$ has a subquiver of the form
$$
\hskip .5cm 
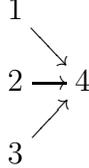
\xymatrix@-4mm{
{1}\ar [dr]\\
{2}\ar[r]&{4}\ar@{<-} [dl]\\
{3}}\hskip .5cm
$$
which gives a contradiction. So one of the maximal submodules $M_1$ and $M_2$ is uniserial and the other one is $3$-factor serial.
\item[$(a)(v)$.] Assume that $M$ is non-local of length $4$ that $soc(M)$ is not simple. By the Proposition \ref{p3}, $l(top(M))=2$ and it implies that $rad(M)=soc(M)=S_3\oplus S_4$ and $top(M)=S_1\oplus S_2$, where $S_i$ is simple for each $i$. $M$ has two maximal submodules $M_1$ and $M_2$ of length $3$
that $rad(M)=soc(M)\subset M_i$ for each $1 \leq i \leq 2$, $ \frac{M}{M_1}\cong S_2 $, $ \frac{M}{M_2}\cong S_1 $ and by \cite[Lemma 5.2]{NS} if $M_i$ is indecomposable for each $1\leq i\leq 2$, then $M_i$ is $2$-factor serial. We show that only one of the $M_1$ and $M_2$ is indecomposable. Assume on the contrary that, $M_1$ and $M_1$ are not indecomposable. $P_1\oplus P_2$ is a projective cover of $M$ and so $\frac{P_1\oplus P_2}{L}\cong M$, where $L$ is a submodule of $P_1\oplus P_2$ and $\frac{rad(P_1\oplus P_2)}{L}\cong rad(M)\cong soc(M)=S_3\oplus S_4$. Therefore each $S_i$, $3\leq i\leq 4$ is a direct summand of one of the $top(rad(P_1))$ and $top(rad(P_2))$. Assume that $S_3$ is direct summand of $top(rad(P_1))$ and $S_4$ is direct summand of $top(rad(P_2))$. Since $M_1$ and $M_2$ are not indecomposable, $S_4$ is not a direct summand of $top(rad(P_1))$ and $S_3$ is not a direct summand of $top(rad(P_2))$. Then $M$ is a direct summand of two uniserial modules of length $2$ which gives a contradiction. Now assume that $M_1$ and $M_2$ are both indecomposable $2$-factor serial right
$\Lambda$-modules. Since for every $1\leq i \leq 2$, $S_1\oplus S_2=rad(M)\subset M_i$ and $P_1, P_2$ are projective covers of $M_1, M_2$, respectively. Then $S_3$ and $S_4$ are direct summands of $top(rad(P_i))$ for each $1\leq i \leq 2$ and by \cite[Proposition III.1.15]{ARS}, the valued quiver of $\Lambda$ has a subquiver of the form
$$\hskip .5cm 
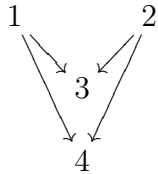
\xymatrix@-4mm{
{1}\ar [dr]\ar [ddr]&&{2}\ar[dl]\ar [ddl]\\
&{3}\\
&{4}}\hskip .5cm$$
Therefore $\Lambda$ is not representation-finite which gives a contradiction. Then we can assume that $M_1$ is indecomposable and $2$-factor serial and $M_2=N\oplus S$ that $N$ is uniserial of length $2$ and $S$ is a simple right $\Lambda$-module.
  \end{proof}

Now by using the Theorem \ref{T2}, we give a characterization  of indecomposable non-projective modules over right $4$-Nakayama artin algebras.

 \begin{theorem}\label{T3}
 Let $\Lambda$ be a right $4$-Nakayama artin algebra and $M$ be an indecomposable right $\Lambda$-module. Then $M$ is non-projective if and only if one of the following situations hold.

 \begin{itemize}
 \item[$(A)$] $M$ is a factor module of an indecomposable projective right $\Lambda$-module $P$ which is uniserial and  $M\cong \frac{P}{rad^{t}(P)}$ for some $1\leq t< l(P)$.
 \item[$(B)$] $M$ is a factor module of an indecomposable projective right $\Lambda$-module $P$ which is a $2$-factor serial right $\Lambda$-module that  $rad(P)=soc(P)=S_1\oplus S_2$ and  $M$ is isomorphic to either $\frac{P}{rad(P)}$ or $ \frac{P}{S_i} $ for some $1\leq i \leq 2$.
 \item[$(C)$]  $M$ is a factor module of an indecomposable projective right $\Lambda$-module $P$ which is a  $3$-factor serial  colocal  right $\Lambda$-module such that  submodules of $P$ are  $rad(P)$ which is $3$-factor serial of length $3$ and non-local, two uniserial submodules $N_1$ and $N_2$ of length $2$ and $soc(P)=S$ which is simple. $M$ is isomorphic to either $\frac{P}{rad(P)}$ or $\frac{P}{N_i}$ for some $1\leq i \leq 2$ or $\frac{P}{S}$.
 \item[$(D)$]  $M$ is a factor module of an indecomposable projective right $\Lambda$-module $P$ which  is a  $3$-factor serial  non-colocal  right $\Lambda$-module such that submodules of $P$ are  $rad(P)=N\oplus S_1$ that $N$ is uniserial of length $2$, $S_1$ is simple and $soc(P)=S_1\oplus S_2$ that $soc(N)=S_2$. $M$ is isomorphic to either $\frac{P}{rad (P)}$ or $\frac{P}{N}$ or $\frac{P}{soc(P)}$ or $\frac{P}{S_i}$ for some $1\leq i \leq 2$.
 \item[$(E)$] $M$ is a factor module of an indecomposable projective right $\Lambda$-module $P$ which is  colocal  $4$-factor serial right $\Lambda$-module such that submodules of $P$ are  $rad(P)$ which  is $4$-factor serial  of length $4$ and non-local, two submodules $M_1$ and $M_2$ of length $3$ that $M_1$ is uniserial and $M_2$ is non-local  $3$-factor serial, two uniserial submodules $N_1$ and $N_2$ of length $2$ that $N_1$ is a submodule of $M_1$ and $M_2$ and $N_2$ is a submodule of $M_2$ and $soc(P)=S$ which is simple. $M$ is isomorphic to either  $\frac{P}{rad(P)}$ or $\frac{P}{M_i}$ for some $1\leq i \leq 2$ or $\frac{P}{N_i}$ for some $1\leq i \leq 2$ or $\frac{P}{S}$.
 \item[$(F)$] $M$ is a factor module of an indecomposable projective right $\Lambda$-module $P$ which is a non-colocal $4$-factor serial right $\Lambda$-module such that submodules of $P$ are  $rad(P)=N\oplus S_1$ that $N$ is uniserial of length  $3$ and $S_1$ is a simple right
 $\Lambda$-module, an indecomposable uniserial submodule $N_1$ of  length $2$ and $soc(P)=S_1\oplus S_2$ that $soc(N)=S_2$. $M$ is isomorphic to either $\frac{P}{rad(P)}$ or $ \frac{P}{N} $ or $ \frac{P}{N_1} $ or $ \frac{P}{N_1\oplus S_1} $ or  $\frac{P}{soc(P)}$ or $ \frac{P}{S_i} $ for some $1\leq i \leq 2$.
 \item[$(G)$]  $M$ is a factor module of an indecomposable projective right $\Lambda$-module $P$ which is a  non-colocal $4$-factor serial right $\Lambda$-module such that submodules of $P$ are  $rad(P)=N_1\oplus N_2$ that $N_i$ is uniserial of length $2$  for each $1\leq i \leq 2$ and $soc(P)=S_1\oplus S_2$  that $soc(N_i)=S_i$ for each $1\leq i \leq 2$. $M$ is isomorphic to either $ \frac{P}{rad(P)} $ or
 $ \frac{P}{N_i} $ for some $1\leq i \leq 2$ or $ \frac{P}{N_1\oplus S_2} $ or $ \frac{P}{N_2\oplus S_1} $  or $ \frac{P}{soc(P)} $ or $ \frac{P}{S_i} $ for some $1\leq i \leq 2$.
 \item[$(H)$] $M$ is a non-local  $3$-factor serial right $\Lambda$-module whose  submodules are  two uniserial modules $M_1$ and $M_2$ and $soc(M)=S$ which is simple. Indecomposable quotients of $M$ are $\frac{M}{M_i}$ for $1\leq i \leq 2$.
 \item[$(I)$] $M$ is a non-local $4$-factor serial right $\Lambda$-module with simple socle whose  submodules  are  two maximal  submodules $M_1$ and $M_2$ of length $3$ that $M_1$ is uniserial and $M_2$ is non-local $3$-factor serial, two uniserial submodules $N_1$ and $N_2$ of length $2$ which $N_1$ is a uniserial submodule of $M_1$ and $M_2$ and $N_2$ is a submodule of $M_2$ and $soc(M)=S$, which is simple. Indecomposable quotients of $M$ are $\frac{M}{M_i}$ for $1\leq i \leq 2$ and $\frac{M}{N_2}$.
 \item[$(J)$] $M$ is a non-local  $4$-factor serial right $\Lambda$-module that $soc(M)$ is not simple. Submodules of $M$ are, an indecomposable module $N_1$ of length $3$ which is $2$-factor serial, a uniserial module $N_2$ of length $2$ and $soc(M)=S_1\oplus S_2$ that $soc(N_2)=S_2$. Indecomposable quotients of $M$ are $ \frac{M}{N_i} $ for $1\leq i \leq 2$, $\frac{M}{N_2\oplus S_1} $ and $\frac{M}{S_1} $.
 \end{itemize}

  \end{theorem}

\section{almost split sequences of right $4$-Nakayama artin algebras}

Now we compute almost split sequences of right $4$-Nakayama artin algebras.

 \begin{theorem}\label{T4}
  Let $\Lambda$ be a right $4$-Nakayama artin algebra and $M$ be an indecomposable non-projective right $\Lambda$-module. Then one of the following situations hold.

\item[$(A)$] Assume that $M\cong \frac{P}{rad^i(P)}$, where $P$ is uniserial projective, $1\leq i <l(P)$ and $M$ is not isomorphic to the quotient of an indecomposable non-local $3$-factor serial module. Also assume that $M$ is not isomorphic to the either $\frac{L}{M_1}$ or $\frac{L}{N_2}$, where $L$ is an indecomposable colocal non-local $4$-factor serial, $M_1$ is a uniserial submodule of $L$ and $N_2$ is a uniserial submodule of $L$ which is not a submodule of $M_1$. Then the following sequence is an almost split sequence.
  \begin{center}
  $0\longrightarrow \frac{rad (P)}{rad^{i+1}(P)} \buildrel{
  \begin{bmatrix}
\pi_1\\
i_1
  \end{bmatrix}
   }\over\longrightarrow \frac{rad (P)}{rad^{i}(P)}\oplus \frac{P}{rad^{i+1}(P)} \buildrel{[- i_{2}, \pi_2]}\over\longrightarrow \frac{P}{rad^{i}(P)}\longrightarrow 0$
  \end{center}
 \item[$(B)$] Assume that $M$ is a factor of an indecomposable projective $2$-factor serial right $\Lambda$-module $P$. Submodules of $P$ are $rad(P)=soc(P)=S_1\oplus S_2$ that $S_i$ is a simple submodule of $P$ for each $1\leq i\leq 2$. Also assume that $P$ is not a submodule of an indecomposable non-local $4$-factor serial right $\Lambda$-module $L$ that $soc(L)$ is not simple.
 \begin{itemize}
 \item[$(i)$]If $M\cong \frac{P}{S_i}$ for some $1\leq i \leq 2$, then the sequence
 \begin{center}
 $0\longrightarrow S_i \buildrel{i_3}\over \longrightarrow P\buildrel{\pi_{3}}\over \longrightarrow \frac{P}{S_i}\longrightarrow0$
 \end{center}
 is an almost split sequence.

 \item[$(ii)$] If $M\cong \frac{P}{rad(P)}$, then the sequence
 \begin{center}
 $0\longrightarrow P\buildrel{
  \begin{bmatrix}
\pi_{4}\\
\pi_{5}
  \end{bmatrix}
   }\over \longrightarrow \frac{P}{S_1}\oplus \frac{P}{S_2}\buildrel{[-\pi_{6}, \pi_{7}]}\over \longrightarrow \frac{P}{rad (P)}\longrightarrow0$
  \end{center}
 is an almost split sequence.
 \end{itemize}

 \item[$(C)$] Assume that $M$ is a factor of an indecomposable projective $3$-factor serial colocal right $\Lambda$-module $P$. Submodules of $P$ are $rad(P)$ which is indecomposable non-local $3$-factor serial of length $3$, two uniserial modules $M_1$ and $M_2$ of length $2$ and $S=soc(P)$ which is simple.
 \begin{itemize}
\item[$(i)$] If $M\cong \frac{P}{rad(P)}$, then the sequence
 \begin{center}
 $ 0\longrightarrow\frac{P}{S}\buildrel{
  \begin{bmatrix}
\pi_{8}\\
\pi_{9}
  \end{bmatrix}
   }\over \longrightarrow \frac{P}{M_1}\oplus\frac{P}{M_2}\buildrel{[-\pi_{10}, \pi_{11}]}\over \longrightarrow \frac{P}{rad(P)}\longrightarrow0 $
 \end{center}
 is an almost split sequence.
 \item[$(ii)$]
 If $M\cong \frac{P}{M_i}$, for some $1\leq i\leq 2$, then the sequence
 \begin{center}
 $ 0\longrightarrow\frac{M_i}{S}\buildrel{i_4}\over \longrightarrow \frac{P}{S}\buildrel{\pi_{12}}\over \longrightarrow \frac{P}{M_i}\longrightarrow0 $
 \end{center}
 is an almost split sequence.

 \item[$(iii)$]
  If $M\cong \frac{P}{S}$, then the sequence
 \begin{center}
 $ 0\longrightarrow rad(P)\buildrel{
  \begin{bmatrix}
\pi_{13}\\
\pi_{14}\\
 i_{5}
  \end{bmatrix}
   }\over \longrightarrow\frac{M_1}{S}\oplus\frac{M_2}{S}\oplus P\buildrel{[-i_{6}, i_{7}, \pi_{15}]}\over \longrightarrow \frac{P}{S}\longrightarrow0 $
 \end{center}
 is an almost split sequence.
  \end{itemize}
 \item[$(D)$] Assume that $M$ is a factor of an indecomposable projective $3$-factor serial non-colocal right $\Lambda$-module $P$. Submodules of $P$ are $rad(P)=N\oplus S_1$ that $N$ is uniserial of length $2$, $S_1$ is simple and $soc(P)=S_1\oplus S_2$ that $soc(N)=S_2$.
\begin{itemize}
\item[$(i)$] If $M\cong \frac{P}{rad(P)}$, then the sequence
\begin{center}
$0\longrightarrow \frac{P}{S_2}\buildrel{
  \begin{bmatrix}
\pi_{16}\\
\pi_{17}
  \end{bmatrix}   }\over\longrightarrow\frac{P}{N}\oplus \frac{P}{soc(P)}\buildrel{[-\pi_{18}, \pi_{19}] }\over\longrightarrow \frac{P}{rad(P)}\longrightarrow0$
\end{center}
is an almost split sequence.
\item[$(ii)$] If $M\cong \frac{P}{N}$, then the sequence
\begin{center}
$ 0\longrightarrow \frac{N}{S_2}\buildrel{i_{8} }\over \longrightarrow \frac{P}{S_2}\buildrel{\pi_{20} }\over \longrightarrow \frac{P}{N}\longrightarrow 0$
\end{center}
is an almost split sequence.
\item[$(iii)$] If $M\cong \frac{P}{soc(P)}$, then the sequence
\begin{center}
$  0\longrightarrow P\buildrel{
  \begin{bmatrix}
\pi_{21}\\
\pi_{22}
  \end{bmatrix}
   }\over\longrightarrow \frac{P}{S_1}\oplus \frac{P}{S_{2}}\buildrel{[-\pi_{23}, \pi_{24}]}\over\longrightarrow \frac{P}{soc(P)}\longrightarrow0$
\end{center}
is an almost split sequence.
\item[$(iv)$] If $M\cong \frac{P}{S_1}$, then the sequence
\begin{center}
$ 0\longrightarrow S_1 \buildrel{i_{9} }\over \longrightarrow P\buildrel{\pi_{25} }\over \longrightarrow \frac{P}{S_1}\longrightarrow 0$
\end{center}
is an almost split sequence.
\item[$(v)$] If $M\cong \frac{P}{S_{2}}$, then the sequence
\begin{center}
$ 0\longrightarrow N\buildrel{  \begin{bmatrix}
\pi_{26}\\
 i_{10}
  \end{bmatrix}}\over \longrightarrow \frac{N}{S_2}\oplus P\buildrel{[-i_{11}, \pi_{27}]}\over \longrightarrow \frac{P}{S_{2}}\longrightarrow 0$
\end{center}
is an almost split sequence.
\end{itemize}
\item[$(E)$] Assume that $M$ is a factor of an indecomposable projective colocal $4$-factor serial right $\Lambda$-module $P$. Submodules of $P$ are $rad(P)$ that is non-local $4$-factor serial of length $4$, two submodules $M_1$ and $M_2$ of length $3$ that $M_1$ is uniserial and $M_2$ is non-local $3$-factor serial, two uniserial submodules $N_1$ and $N_2$ of length $2$ that $N_1$ is a submodule of $M_1$ and $M_2$, $N_2$ is a submodule of $M_2$ and $soc(P)=S$ which is simple.
\begin{itemize}
\item[$(i)$] If $M\cong \frac{P}{rad(P)}$, then the sequence
\begin{center}
$0\longrightarrow \frac{P}{N_1}\buildrel{  \begin{bmatrix}
\pi_{28}\\
\pi_{29}
  \end{bmatrix}}\over\longrightarrow \frac{P}{M_1}\oplus\frac{P}{M_2}\buildrel{[-\pi_{30}, \pi_{31}]}\over \longrightarrow\frac{P}{rad (P)}\longrightarrow 0$
\end{center}
 is an almost split sequence.
 \item[$(ii)$] If $M\cong \frac{P}{M_2}$, then the sequence
 \begin{center}
 $0\longrightarrow\frac{P}{S}\buildrel{  \begin{bmatrix}
\pi_{32}\\
\pi_{33}
  \end{bmatrix}}\over\longrightarrow \frac{P}{N_1}\oplus \frac{P}{N_2}\buildrel{[-\pi_{34}, \pi_{35}]}\over\longrightarrow \frac{P}{M_2}\longrightarrow0$
 \end{center}
  is an almost split sequence.
  \item[$(iii)$]  If $M\cong \frac{P}{M_1}$, then the sequence
  \begin{center}
$0\longrightarrow \frac{M_1}{N_1} \buildrel{i_{12}}\over \longrightarrow \frac{P}{N_1}\buildrel{\pi_{36}}\over\longrightarrow \frac{P}{M_1}\longrightarrow 0$
  \end{center}
  is an almost split sequence.
  \item[$(iv)$] If $M\cong \frac{P}{N_1}$, then the sequence
  \begin{center}
  $0\longrightarrow \frac{M_1}{S}\buildrel{  \begin{bmatrix}
\pi_{37}\\
i_{13}
  \end{bmatrix}}\over \longrightarrow \frac{M_1}{N_1}\oplus \frac{P}{S}\buildrel{[-i_{14}, \pi_{38}]}\over\longrightarrow \frac{P}{N_1}\longrightarrow 0$
  \end{center}
  is an almost split sequence.
  \item[$(v)$] If $M\cong \frac{P}{N_2}$, then the sequence
  \begin{center}
  $0\longrightarrow \frac{N_2}{S}\buildrel {i_{15}}\over \longrightarrow \frac{P}{S}\buildrel {\pi_{39}}\over\longrightarrow \frac{P}{N_2}\longrightarrow 0$
  \end{center}
  is an almost split sequence.
  \item[$(vi)$] If $M\cong \frac{P}{S}$, then the sequence
  \begin{center}
  $0\longrightarrow rad(P)\buildrel{  \begin{bmatrix}
i_{16}\\
\pi_{40}\\
\pi_{41}
  \end{bmatrix}}\over\longrightarrow P\oplus \frac{M_1}{S} \oplus\frac{N_2}{S}\buildrel{[\pi_{42}, -i_{17}, i_{18}]}\over \longrightarrow \frac{P}{S} \longrightarrow 0 $
  \end{center}
  is an almost split sequence.
\end{itemize}
\item[$(F)$] Assume that $M$ is a factor of an indecomposable projective non-colocal $4$-factor serial right $\Lambda$-module $P$. Submodules of $P$ are $rad(P)=N\oplus S_1$ that $N$ is a uniserial module of length $3$ and $S_1$ is a simple right $\Lambda$-module, an indecomposable uniserial submodule $N_1$ of length $2$ and $soc(P)=S_1\oplus S_2$ that $S_2=soc(N)$.
\begin{itemize}
\item[$(i)$] If $M\cong \frac{P}{rad (P)}$, then the sequence
\begin{center}
$0\longrightarrow \frac{P}{N_1}\buildrel{  \begin{bmatrix}
\pi_{43}\\
\pi_{44}
  \end{bmatrix}}\over\longrightarrow \frac{P}{N_1\oplus S_1}\oplus \frac{P}{N}\buildrel {[-\pi_{45}, \pi_{46}]}\over\longrightarrow\frac{P}{rad (P)}\longrightarrow 0$
\end{center}
is an almost split sequence.
\item[$(ii)$] If $M\cong \frac{P}{N}$, then the sequence
\begin{center}
$0\longrightarrow \frac{N}{N_1} \buildrel {i_{19}}\over \longrightarrow \frac{P}{N_1}\buildrel{\pi_{47}}\over \longrightarrow \frac{P}{N} \longrightarrow 0$
\end{center}
is an almost split sequence.
\item[$(iii)$] If $M\cong \frac{P}{N_1\oplus S_1}$, then the sequence
\begin{center}
$0\longrightarrow \frac{P}{S_2}\buildrel{  \begin{bmatrix}
\pi_{48}\\
\pi_{49}
  \end{bmatrix}}\over\longrightarrow \frac{P}{soc(P)}\oplus \frac{P}{N_1}\buildrel {[-\pi_{50}, \pi_{51}]}\over\longrightarrow   \frac{P}{N_1\oplus S_1}\longrightarrow 0$
\end{center}
is an almost split sequence.
\item[$(iv)$] If $M\cong \frac{P}{ soc(P)}$, then the sequence
\begin{center}
$0\longrightarrow P \buildrel{  \begin{bmatrix}
\pi_{52}\\
\pi_{53}
  \end{bmatrix}}\over\longrightarrow \frac{P}{S_1}\oplus \frac{P}{S_2}\buildrel {[-\pi_{54}, \pi_{55}]}\over \longrightarrow \frac{P}{ soc(P)} \longrightarrow 0$
\end{center}
is an almost split sequence.
\item[$(v)$] If $M\cong \frac{P}{ N_1}$, then the sequence
\begin{center}
$0\longrightarrow \frac{N}{S_2}\buildrel{  \begin{bmatrix}
i_{20}\\
\pi_{56}
  \end{bmatrix}}\over\longrightarrow \frac{P}{S_2}\oplus \frac{N}{N_1} \buildrel{[-\pi_{57}, i_{21}]}\over\longrightarrow \frac{P}{ N_1} \longrightarrow 0 $
\end{center}
is an almost split sequence.
\item[$(vi)$] If $M\cong \frac{P}{S_1}$, then the sequence
\begin{center}
$0\longrightarrow S_1\buildrel{i_{22}}\over \longrightarrow P\buildrel{\pi_{58}}\over \longrightarrow \frac{P}{S_1}\longrightarrow 0$
\end{center}
is an almost split sequence.
\item[$(vii)$] If $ M\cong \frac{P}{S_2} $, then the sequence
\begin{center}
$0\longrightarrow N\buildrel{  \begin{bmatrix}
i_{23}\\
\pi_{59}
  \end{bmatrix}} \over \longrightarrow P\oplus \frac{N}{S_2}\buildrel{[-\pi_{60}, i_{24}]}\over\longrightarrow \frac{P}{S_2}\longrightarrow 0$
\end{center}
is an almost split sequence.
\end{itemize}
\item[$(G)$] Assume that $M$ is a factor of an indecomposable projective non-colocal $4$-factor serial right $\Lambda$-module $P$. Submodules of $P$ are $rad(P)=N_1\oplus N_2$ that for each $1\leq i \leq 2$, $N_i$ is uniserial of length $2$ and $soc(P)=S_1\oplus S_2$ that for each $1\leq i \leq 2$, $soc(N_i)=S_i$.
\begin{itemize}
\item[$(i)$] If $M\cong \frac{P}{rad (P)}$, then the sequence
\begin{center}
$0 \longrightarrow \frac{P}{soc (P)}\buildrel{  \begin{bmatrix}
\pi_{61}\\
\pi_{62}
  \end{bmatrix}} \over \longrightarrow \frac{P}{N_1\oplus S_2}\oplus \frac{P}{N_2\oplus S_1}\buildrel{[-\pi_{63}, \pi_{64}]}\over\longrightarrow \frac{P}{rad(P)}\longrightarrow 0$
\end{center}
is an almost split sequence.
\item[$(ii)$] If $M\cong \frac{P}{N_i\oplus S_j}$ for some $1\leq i,j \leq 2$ and $i\neq j$, then the sequence
\begin{center}
$0\longrightarrow \frac{P}{S_i}\buildrel{  \begin{bmatrix}
\pi_{65}\\
\pi_{66}
  \end{bmatrix}} \over\longrightarrow \frac{P}{soc(P)}\oplus \frac{P}{N_i}\buildrel{[-\pi_{67}, \pi_{68}]}\over \longrightarrow \frac{P}{N_i\oplus S_j}\longrightarrow 0 $
\end{center}
is an almost split sequence.
\item[$(iii)$] If $M\cong \frac{P}{N_i}$ for some $1\leq i \leq 2$, then the sequence
\begin{center}
$0 \longrightarrow \frac{N_i}{S_i}\buildrel{i_{25}}\over \longrightarrow \frac{P}{S_i}\buildrel {\pi_{69}}\over\longrightarrow   \frac{P}{N_i} \longrightarrow 0$
\end{center}
is an almost split sequence.
\item[$(iv)$] If $M\cong \frac{P}{soc( P)}$, then the sequence
\begin{center}
$0\longrightarrow P\buildrel{  \begin{bmatrix}
\pi_{70}\\
\pi_{71}
  \end{bmatrix}} \over \longrightarrow \frac{P}{S_1}\oplus \frac{P}{S_2}\buildrel{[-\pi_{72}, \pi_{73}]}\over\longrightarrow \frac{P}{soc (P)}\longrightarrow 0$
\end{center}
is an almost split sequence.
\item[$(v)$] If $M\cong \frac{P}{S_i}$ for some $1\leq i \leq 2$, then the sequence
\begin{center}
$0\longrightarrow N_i\buildrel{  \begin{bmatrix}
\pi_{74}\\
i_{26}
  \end{bmatrix}} \over \longrightarrow \frac{N_i}{S_i} \oplus P \buildrel{[-i_{27}, \pi_{75}]}\over \longrightarrow  \frac{P}{S_i} \longrightarrow 0$
\end{center}
is an almost split sequence.
\end{itemize}
\item[$(H)$] Assume that $M$ is an indecomposable non-local $3$-factor serial right $\Lambda$-module which is not isomorphic to the quotient of an indecomposable non-local non-colocal $4$-factor serial right $\Lambda$-module $L$. $M$ is of length $3$ and submodules of $M$ are two uniserial maximal submodules $M_1$ and $M_2$ of length $2$ and $rad (M)=soc(M)=S$ which is simple.
\begin{itemize}
\item[$(i)$] The exact sequence
\begin{center}
$ 0\longrightarrow S\buildrel{  \begin{bmatrix}
i_{29}\\
 i_{30}
  \end{bmatrix}}\over\longrightarrow M_{1}\oplus M_{2}\buildrel{ [-i_{31}, i_{32}]}\over \longrightarrow M\longrightarrow 0$
\end{center}
is an almost split sequence.
\item[$(ii)$] For an indecomposable right $\Lambda$-module $\frac{M}{M_j}$, $1\leq j \leq 2$, the exact sequence
\begin{center}
$0\longrightarrow M_j \buildrel{ i_{28} }\over\longrightarrow M \buildrel{ \pi_{76} }\over\longrightarrow \frac{M}{M_ {j}}\longrightarrow0  $
\end{center}
is an almost split sequence.
\end{itemize}
\item[$(I)$] Assume that $M$ is an indecomposable non-local $4$-factor serial right $\Lambda$-module that $soc(M)$ is simple. Submodules of $M$ are two indecomposable submodules $M_1$ and $M_2$ of length $3$ that $M_1$ is uniserial and $M_2$ is non-local, two uniserial submodules $N_1$ and $N_2$ of length $2$ that $N_1$ is a submodule of $M_1$ and $M_2$, $N_2$ is a submodule of $M_2$ and $soc(M)=S$ which is simple.
\begin{itemize}
\item[$(i)$] The exact sequence
\begin{center}
$0\longrightarrow N_1\buildrel{  \begin{bmatrix}
i_{33}\\
i_{34}
  \end{bmatrix}} \over \longrightarrow M_1\oplus M_2\buildrel{[-i_{35}, i_{36}]}\over\longrightarrow M\longrightarrow 0$
\end{center}
is an almost split sequence.
\item[$(ii)$] For an indecomposable module $\frac{M}{M_1}$, the exact sequence
$$0\longrightarrow M_1\buildrel{i_{37}}\over\longrightarrow M  \buildrel{\pi_{77}}\over \longrightarrow  \frac{M}{M_1}\longrightarrow 0$$
   is an almost split sequence.
\item[$(iii)$] For an indecomposable module $\frac{M}{N_2}$, the exact sequence
$$0\longrightarrow M_2\buildrel{\begin{bmatrix}
\pi_{78}\\
i_{38}
  \end{bmatrix}
   }\over\longrightarrow soc(\frac{M}{N_2})\oplus M \buildrel{[-i_{39}, \pi_{79}]}\over \longrightarrow  \frac{M}{N_2} \longrightarrow 0$$
   is an almost split sequence.
   \end{itemize}
\item[$(J)$] Assume that $M$ is an indecomposable non-local $4$-factor serial right $\Lambda$-module of length $4$ that $soc(M)$ is not simple. Submodules of $M$ are $M_1$ and $M_2$ that $M_1$ is $2$-factor serial of length $3$, $M_2$ is uniserial of length $2$ and $soc(M)=S_1\oplus S_2$ that $soc(M_2)=S_2$.
\begin{itemize}
\item[$(i)$] The sequence
\begin{center}
$0\longrightarrow S_2\buildrel{  \begin{bmatrix}
\pi_{80}\\
i_{40}
  \end{bmatrix}}\over \longrightarrow M_1\oplus M_2\buildrel{[-i_{41}, \pi_{81}]}\over\longrightarrow M\longrightarrow 0$
\end{center}
is an almost split sequence.
\item[$(ii)$] For an indecomposable module $\frac{M}{S_1}$ the sequence
\begin{center}
$0\longrightarrow M_1\buildrel{  \begin{bmatrix}
\pi_{82}\\
i_{42}
  \end{bmatrix}}\over \longrightarrow \frac{M_1}{S_1}\oplus M \buildrel{[-i_{43}, \pi_{83}]}\over \longrightarrow \frac{M}{S_1}\longrightarrow 0$
\end{center}
is an almost split sequence.

\item[$(iii)$] For an indecomposable module $\frac{M}{M_2}$, the sequence
  $$0\longrightarrow M_2\buildrel{i_{44}}\over \longrightarrow M \buildrel{\pi_{84}}\over \longrightarrow  \frac{M}{M_2}\longrightarrow 0 $$
  is an almost split sequence.
 \item[$(iv)$] For an indecomposable module $\frac{M}{M_2 \oplus S_1}$, the sequence
 $$0\longrightarrow M \buildrel{\begin{bmatrix}
\pi_{85}\\
\pi_{86}
  \end{bmatrix}}\over\longrightarrow \frac{M}{S_1}\oplus \frac{M}{M_2}\buildrel{[-\pi_{87}, \pi_{88}]}\over\longrightarrow \frac{M}{M_2 \oplus S_1} \longrightarrow 0$$
is an almost split sequence.
\item[$(v)$] For an indecomposable module $\frac{M}{M_1}$, the sequence
 $$0\longrightarrow \frac{M_1}{S_1}\buildrel{i_{45}}\over \longrightarrow \frac{M}{S_1} \buildrel{\pi_{89}}\over \longrightarrow  \frac{M}{M_1}\longrightarrow 0 $$
is an almost split sequence.

\end{itemize}
Where $i_j $ is an inclusion for each $1\leq j \leq 45$ and $\pi_j $ is a canonical epimorphism for each $1\leq j \leq 89$.
\end{theorem}
\begin{proof}
We only  prove the parts $(H)(i)$, $(H)(ii)$, $(I)(i)$, $(I)(ii)$, $(I)(iii)$, $(J)(i)$, $(J)(ii)$, $(J)(iii)$, $(J)(iv)$ and $(J)(v)$. In the other parts, since $M$ is a quotient of an indecomposable projective right $\Lambda$-module without any condition, the proof is easy. Put $g_1=[-i_{31}, i_{32}]$, $g_2=[-i_{35}, i_{36}]$,  $g_3=[-i_{39}, \pi_{79}]$,  $g_4=[-i_{41}, \pi_{81}]$,  $g_5=[-i_{43}, \pi_{83}]$ and $g_6=[-\pi_{87}, \pi_{88}]$. It is easy to see that all given sequences are exact, non-split and have indecomposable end terms. Then it is enough to show that homomorphisms $g_1$, $\pi_{76}$, $g_2$, $\pi_{77}$, $g_3$, $g_4$, $g_{5}$, $ \pi_{84} $, $g_6$ and $\pi_{89}$ are right almost split morphisms.
\item[$(H)(i).$] Let $V$ be an indecomposable right $\Lambda$-module and $\nu:V\longrightarrow M$ be a non-isomorphism. If $\nu$ is an epimorphism, then by \cite[Lemma 2.11]{NS}, $V$ is a $4$-factor serial right $\Lambda$-module and by the Theorem \ref{T3} $V$ is non-local non-colocal 4-factor serial. Then $M$ isomorphic to the quotient of an indecomposable non-local non-colocal 4-factor serial right $\Lambda$-module which is a contradiction. Thus $Im(\nu)$ is a proper submodule of $M$. Since $M_1$ and $M_2$ are maximal submodules of $M$, $Im(\nu)$ is a submodule of $M_1\oplus M_2$. Then there exists a homomorphism $h:V\longrightarrow M_1 \oplus M_2$ such that $\nu=g_1 h$.
\item[$(H)(ii).$] Let $V$ be an indecomposable right $\Lambda$-module and $\nu:V\longrightarrow \frac{M}{M_j}$ be a non-isomorphism. Since $\frac{M}{M_j}$ is a simple right $\Lambda$-module, $\nu$ is an epimorphism. If $j=1$, then $V$ is isomorphic to either $M$ or $M_2$ or a quotient of an indecomposable projective right $\Lambda$-module $P$. If $V$ is isomorphic to either $M$ or $M_2$, then there exists a homomorphism $h: V\longrightarrow M$ such that $\pi_{76} h=\nu$. Now assume that $V$ is isomorphic to the quotient of $P$. We show that in this case $P$ is uniserial. Let $ \frac{M}{M_1}\cong S_2 $, $top(M)=S_1\oplus S_2$ and $ rad(M)=soc(M)=S_3 $, then $P_1\oplus P_2$ is a projective cover of $M$ and $P=P_2$. By \cite[Proposition III.1.15]{ARS}, the valued quiver of $\Lambda$ has a subquiver of the form
$$\hskip .5cm 
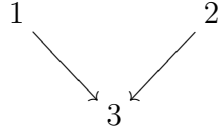
\xymatrix{
{1}\ar  [dr]&&{2}\ar [dl] \\
&{3}
}\hskip .5cm$$
If $P_2$ is not uniserial, then by Theorem \ref{T2}, $top(rad(P_2))=S_3\oplus S_4$ and by \cite[Proposition III.1.15]{ARS}, the valued quiver of $\Lambda$ has a subquiver of the form
$$\hskip .5cm 
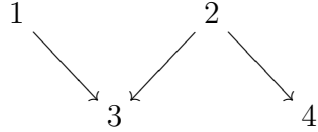
\xymatrix{
{1}\ar [dr]&&{2}\ar [dl] \ar  [dr]\\
&{3} &&{4}
}\hskip .5cm$$
Thus $M$ is isomorphic to the quotient of an indecomposable non-local non-colocal $4$-factor serial right $\Lambda$-module $L$, which gives a contradiction. Then $P$ is uniserial. Also $M_2$ is a submodule of $M$ and $top(\frac{P}{rad^{i}(P)})$ is direct summand of $top(M)$ for some $2\leq i < l(P)$. Therefore there exists a homomorphism $h: V\longrightarrow M$ such that $\pi_{76} h=\nu$.
\item[$(I)(i).$] The proof is similar to the proof of $(H)(i)$.
\item[$(I)(ii).$] Let $V$ be an indecomposable right $\Lambda$-module and $\nu:V\longrightarrow \frac{M}{M_1}$ be a non-isomorphism. Since $\frac{M}{M_1}$ is simple, $V$ is isomorphic to either $M$ or $M_2$ or $N_2$ or a quotient of an indecomposable projective right $\Lambda$-module $P$. The similar argument as in the proof of $(H)(ii)$ shows that $P$ is uniserial. Then there exists a homomorphism $h:V\longrightarrow M$ such that $\pi_{77}h=\nu$.
\item[$(I)(iii).$] Let $V$ be an indecomposable right $\Lambda$-module and $\nu:V\longrightarrow \frac{M}{N_2} $ be non-isomorphism. If $\nu$ is an epimorphism, then $V$ is isomorphic to either $M$ or $M_1$ or a quotient of an indecomposable projective right $\Lambda$-module $P$. The similar argument as in the proof of $(H)(ii)$ shows that $P$ is uniserial. Since $top(\frac{P}{rad^{i}(P)})\cong top(\frac{M}{N_2})$, there exists a homomorphism $h: V\longrightarrow soc( \frac{M}{N_2} )\oplus M$ such that $g_{3}h=\nu$. If $\nu$ is not an epimorphism, then $Im \nu=soc(\frac{M}{N_2})$ and so there exists a homomorphism $h: V\longrightarrow soc(\frac{M}{N_2})\oplus M$ such that $g_{3}h=\nu$.
\item[$(J)(i).$] The proof is similar to the proof of $(H)(i)$.
\item[$(J)(ii).$] Let $V$ be an indecomposable right $\Lambda$-module and $\nu: V \longrightarrow  \frac{M}{S_1}$ be a non-isomorphism. If $\nu$ is an epimorphism, then $V\cong M$ and there exists a homomorphism $h: V \longrightarrow \frac{M_1}{s_1}\oplus M $ such that $g_{5}h=\nu$. Now assume that $\nu$ is not an epimorphism. Maximal submodules of $\frac{M}{S_1}$ are $M_2$ and $\frac{M_1}{S_1}$, then there exists a homomorphism $h:V \longrightarrow \frac{M_1}{s_1}\oplus M$ such that $g_{5}h=\nu$.
\item[$(J)(iii).$] Let $V$ be an indecomposable right $\Lambda$-module and $\nu:V\longrightarrow \frac{M}{M_2} $ be a non-isomorphism. If $\nu$ is an epimorphism, then $V$ is isomorphic to either $M$ or $M_1$ or a quotient of an indecomposable projective right $\Lambda$-module $P$. The similar argument as in the proof of $(H)(ii)$ shows that $M_1$ is a quotient of $P$. Then there exists a homomorphism $h:V\longrightarrow M$ such that $\pi_{84} h=\nu$. Now assume that $\nu$ is not an epimorphism. Since $ Im(\nu)=soc(\frac{M}{M_2}) \subset soc(M)$, $Im(\nu)$ is a submodule of $M$ and so there exists a homomorphism $h:V\longrightarrow M$ such that $\pi_{84} h=\nu$.
\item[$(J)(iv).$] Let $V$ be an indecomposable right $\Lambda$-module and $\nu:V\longrightarrow \frac{M}{M_2 \oplus S_1} $ be a non-isomorphism. Since $ \frac{M}{M_2 \oplus S_1}$ is simple, $V$ is  isomorphic to either $M$  or $M_1$ or $ \frac{M}{S_1}$ or $ \frac{M}{M_2}$ or a quotient of an indecomposable projective right $\Lambda$-module $P$. The similar argument as in the proof of $(H)(ii)$ shows that $M_1$ is a quotient of $P$. then there exists a homomorphism $h:V\longrightarrow \frac{M}{S_1}\oplus \frac{M}{M_2}$ such that $g_6  h=\nu$.
\item[$(J)(v).$] The proof is similar to the proof of $(J)(iii)$.
\end{proof}

\section{quivers and relations of right 4-Nakayama finite dimensional $K$-algebras}

Let $\Lambda$ be a basic connected finite dimensional $K$-algebra. It is known that there exist a quiver $Q$ and an admissible ideal $I$ of the path algebra $KQ$ such that $\Lambda\cong\frac{KQ}{I}$. In this section, we give a necessary and sufficient conditions for the quiver $Q$ and the admissible ideal $I$ that $\frac{KQ}{I}$ be a right $4$-Nakayama algebra.\\
 A finite dimensional $K$-algebra $\Lambda=\frac{KQ}{I}$ is called special biserial algebra provided $(Q,I)$ satisfying the following conditions:
\begin{itemize}
\item[$(1)$] For any vertex $a\in Q_0$, $|a^+|\leq 2$ and   $|a^-|\leq 2$.
\item[$(2)$] For any arrow $\alpha \in Q_1$, there is at most one arrow $\beta$ and at most one arrow $\gamma$ such that $\alpha\beta$ and $\gamma\alpha$ are not in $I$.
\end{itemize}
Let $\Lambda=\frac{KQ}{I}$ be a special biserial finite dimensional $K$-algebra. A walk $w=c_1c_2\cdots c_n$ is called string of length $n$ if $c_i\neq c_{i+1}^{-1}$ for each $i$ and no subwalk nor its inverse is in $I$. In addition, we have strings of length zero, for any $a\in Q_0$ we have two strings of length zero, denoted by $1_{(a,1)}$ and $1_{(a,-1)}$. We have $s(1_{(a,1)})=t(1_{(a,1)})=s(1_{(a,-1)})=t(1_{(a,-1)})=a$ and $1_{(a,1)}^{-1}=1_{(a,-1)}$. A string $w=c_1c_2\cdots c_n$ with $s(w)=t(w)$ such that each
power $w^m$ is a string, but $w$ itself is not a proper power of any strings is called band. We denote by $\mathcal{S}(\Lambda)$ and $\mathcal{B}(\Lambda)$ the set of all strings of $\Lambda$ and the set of all bands of $\Lambda$, respectively. Let $\rho$ be the equivalence relation on $\mathcal{S}(\Lambda)$ which identifies every string $w$ with its inverse $w^{-1}$ and $\sigma$ be the equivalence relation on $\mathcal{B}(\Lambda)$ which identifies every band $w=c_1c_2\cdots c_n$ with the cyclically permuted bands $w_{(i)}=c_ic_{i+1}\cdots c_nc_1\cdots c_{i-1}$ and their inverses $w_{(i)}^{-1}$, for each $i$.
Butler and Ringle in \cite{BR} for each string $w$ defined a unique string module $M(w)$ and for each band $v$ defined a family of band modules $M(v,m,\varphi)$ with $m\geq 1$ and $\varphi\in Aut(K^m)$. Let $\widetilde{\mathcal{S}}(\Lambda)$ be the complete set of representatives of strings relative to $\rho$ and $\widetilde{\mathcal{B}}(\Lambda)$  be the complete set of representatives of bands relative to $\sigma$. Butler and Ringle in \cite{BR} proved that, the modules $M(w)$, $w\in \widetilde{\mathcal{S}}(\Lambda)$ and the modules $M(v,m,\varphi)$ with $v\in \widetilde{\mathcal{B}}(\Lambda)$, $m\geq 1$ and $\varphi\in Aut(K^m)$ provide complete list of pairwise non-isomorphic indecomposable $\Lambda$-modules.
Indecomposable right $\Lambda$-modules are either string modules or band modules or non-uniserial projective-injective modules (see \cite{BR} and \cite{WW}). If $\Lambda$ is a representation-finite special biserial algebra, then any indecomposable right $\Lambda$-module is either string or non-uniserial projective-injective.

\begin{proposition}\label{p4}
Any basic connected finite dimensional right $4$-Nakayama $K$-algebra is representation-finite special biserial.
\end{proposition}
\begin{proof}
The proof is similar to the proof of \cite[Proposition 3.2]{NS1}.
\end{proof}
\begin{theorem}\label{T5}
Let $\Lambda=\frac{KQ}{I}$ be a basic, connected and finite dimensional $K$-algebra.
Then $\Lambda$ is a right $4$-Nakayama algebra if and only if $\Lambda$ is a representation-finite special biserial algebra that $\left( Q, I\right) $ satisfying the following conditions:
\begin{itemize}
\item[$(i)$] If there exist a walk $W$ and two different arrows $w_1$ and $w_2$ with the same target such that  $w_1^{+1}w_2^{-1}$ is a subwalk of $W,$ then $length(W)\leq 3$.
\item[$(ii)$] If there exist a walk $W$ and two different arrows $w_1$ and $w_2$ with the same source such that  $w_1^{-1}w_2^{+1}$ is a subwalk of $W,$ then $length(W)\leq 4$.

\item[$(iii)$] If  there exist two paths $p$ and $q$ with the same target and the same source such that  $p-q\in I$. Then $length(p)+length(q)\leq 5$.
\item[$(iv)$]  At least one of the following conditions hold:
\begin{itemize}
\item[$(a)$] There exist a walk $W$ of length $3$ and two different arrows $w_1$ and $w_2$ with the same target such that  $w_1^{+1}w_2^{-1}$ is a subwalk of $W$.
\item[$(b)$] There exist a walk $W$ of length $4$ and two different arrows $w_1$ and $w_2$ with the same source such that  $w_1^{-1}w_2^{+1}$ is a subwalk of $W$.
\item[$(c)$] There exist two paths $p$ and $q$ with the same target and the same source such that $p-q\in I$ and $length(p)+length(q)=5$.
\end{itemize}

\end{itemize}
\end{theorem}
\begin{proof}
Let $\Lambda$ be a right $4$-Nakayama algebra. By Proposition \ref{p4}, $\Lambda$ is a special biserial algebra of finite type. Assume on the contrary that the condition $(i)$ does not hold. Then there exists a walk $W$ of length greater than or equal $4$ which has a subwalk $W^{'}$ of length $3$ that $W^{'}$ has a subwalk of the form
$w_{1}^{+1}w_{2}^{-1}$.
 Since $\Lambda$ is a special biserial algebra of finite type, the walk $W^{'}$ is one of the following forms.

 \begin{itemize}
 \item First case: The walk $W^{'}$ is of the form
   \begin{center}
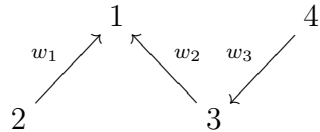
$$\hskip .5cm \xymatrix{
&{1}\ar @{<-}[dl]_{w_1} \ar @{<-} [dr]^{w_2}&&{4}\ar [dl]_{w_3}\\
{2} &&{3}
}\hskip .5cm$$
\end{center}
In this case $W$ has a subwalk of one of the following forms:
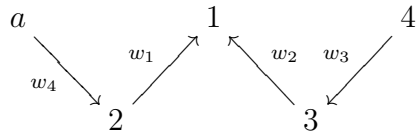
\begin{itemize}
\item[$(i)$]
$$\hskip .5cm \xymatrix{
{a}\ar [dr]_{w_4}&&{1}\ar @{<-}[dl]_{w_1} \ar @{<-} [dr]^{w_2}&&{4}\ar [dl]_{w_3}\\
&{2} &&{3}
}\hskip .5cm$$
In this  case the vertex $a$ can be either $1$ or $2$  or $5$.

\item[$(ii)$]

$$\hskip .5cm 
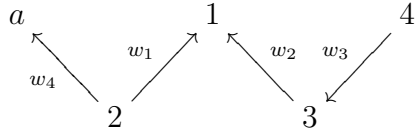
\xymatrix{
{a}\ar @{<-} [dr]_{w_4}&&{1}\ar @{<-}[dl]_{w_1} \ar @{<-} [dr]^{w_2}&&{4}\ar [dl]_{w_3}\\
&{2} &&{3}
}\hskip .5cm$$

In this  case the vertex $a$ can be either $2$ or $3$ or $4$ or $5$.
\item[$(iii)$]
$$\hskip .5cm 
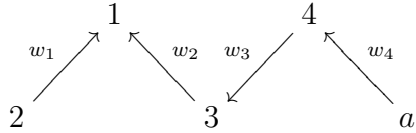
\xymatrix{
&{1}\ar @{<-}[dl]_{w_1} \ar @{<-} [dr]^{w_2}&&{4}\ar [dl]_{w_3}\ar  @{<-} [dr]^{w_4}\\
{2} &&{3}&&{a}
}\hskip .5cm$$
In this case the vertex $a$ can be either $1$  or $3$ or $4$ or $5$.
\item[$(iv)$]
$$\hskip .5cm 
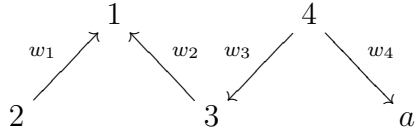
\xymatrix{
&{1}\ar @{<-}[dl]_{w_1} \ar @{<-} [dr]^{w_2}&&{4}\ar [dl]_{w_3}\ar [dr]^{w_4}\\
{2} &&{3}&&{a}
}\hskip .5cm$$
In this case the vertex $a$ can be either $2$  or  $4$ or $5$.
\end{itemize}
 \item Second case: The walk $W^{'}$ is of the form
   \begin{center}
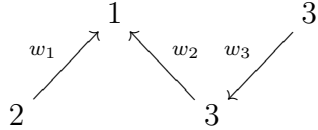
$$\hskip .5cm \xymatrix{
&{1}\ar @{<-}[dl]_{w_1} \ar @{<-} [dr]^{w_2}&&{3}\ar [dl]_{w_3}\\
{2} &&{3}
}\hskip .5cm$$
\end{center}
In this case $W$ has a subwalk of one of the following forms:
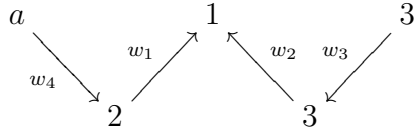
\begin{itemize}
\item[$(i)$]
$$\hskip .5cm \xymatrix{
{a}\ar [dr]_{w_4}&&{1}\ar @{<-}[dl]_{w_1} \ar @{<-} [dr]^{w_2}&&{3}\ar [dl]_{w_3}\\
&{2} &&{3}
}\hskip .5cm$$
In this  case the vertex $a$ can be either $1$ or $2$ or $4$.

\item[$(ii)$]

$$\hskip .5cm 
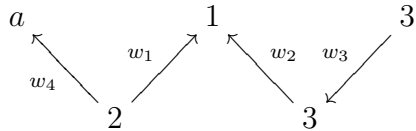
\xymatrix{
{a}\ar @{<-} [dr]_{w_4}&&{1}\ar @{<-}[dl]_{w_1} \ar @{<-} [dr]^{w_2}&&{3}\ar [dl]_{w_3}\\
&{2} &&{3}
}\hskip .5cm$$

In this  case the vertex $a$ can be either $2$  or $4$.
\item[$(iii)$]
$$\hskip .5cm 
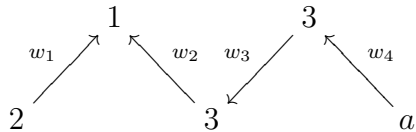
\xymatrix{
&{1}\ar @{<-}[dl]_{w_1} \ar @{<-} [dr]^{w_2}&&{3}\ar [dl]_{w_3}\ar  @{<-} [dr]^{w_4}\\
{2} &&{3}&&{a}
}\hskip .5cm$$
In this case the vertex $a$ can be either  $1$ or $4$.
\end{itemize}
\item Third case: The walk $W^{'}$ is of the form
   \begin{center}
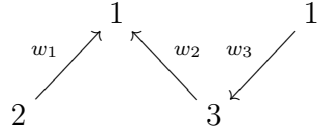
$$\hskip .5cm \xymatrix{
&{1}\ar @{<-}[dl]_{w_1} \ar @{<-} [dr]^{w_2}&&{1}\ar [dl]_{w_3}\\
{2} &&{3}
}\hskip .5cm$$
\end{center}
In this case $W$ has a subwalk of one of the following forms:
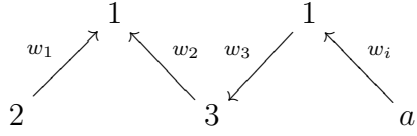
\begin{itemize}
\item[$(i)$]
$$\hskip .5cm \xymatrix{
&{1}\ar @{<-}[dl]_{w_1} \ar @{<-} [dr]^{w_2}&&{1}\ar [dl]_{w_3}\ar@{<-}[dr]^{w_i}\\\
{2} &&{3}&&{a}
}\hskip .5cm$$
for $i=1, 2$. If $i=1$, then $a=2$ and if $i=2$, then $a=3$.
\item[$(ii)$]
$$\hskip .5cm 
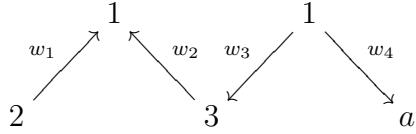
\xymatrix{
&{1}\ar @{<-}[dl]_{w_1} \ar @{<-} [dr]^{w_2}&&{1}\ar [dl]_{w_3}\ar[dr]^{w_4}\\
{2} &&{3}&&{a}
}\hskip .5cm$$
In this case the vertex $a$ can be either $2$ or $4$.
\item[$(iii)$]
$$\hskip .5cm 
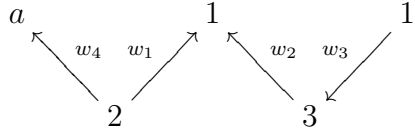
\xymatrix{
{a}\ar@{<-}[dr]^{w_4}&&{1}\ar @{<-}[dl]_{w_1} \ar @{<-} [dr]^{w_2}&&{1}\ar [dl]_{w_3}\\
&{2} &&{3}
}\hskip .5cm$$
In this case the vertex $a$ can be either $2$ or $3$ or $4$.
\item[$(iv)$]
$$\hskip .5cm 
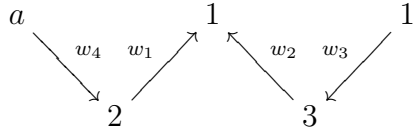
\xymatrix{
{a}\ar[dr]^{w_4}&&{1}\ar @{<-}[dl]_{w_1} \ar @{<-} [dr]^{w_2}&&{1}\ar [dl]_{w_3}\\
&{2} &&{3}
}\hskip .5cm$$
In this case the vertex $a$ can be either $2$ or $4$.
\end{itemize}
\item Forth case: The walk $W^{'}$ is of the form
$$\hskip .5cm 
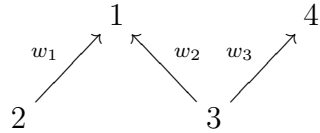
\xymatrix{
&{1}\ar @{<-}[dl]_{w_1} \ar @{<-} [dr]^{w_2}&&{4}\ar @{<-} [dl]_{w_3}\\
{2} &&{3}
}\hskip .5cm$$

In this case $W$ has a subwalk of one of the following forms:
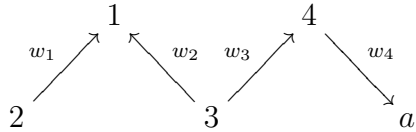
\begin{itemize}
\item[$(i)$]
$$\hskip .5cm \xymatrix{
&{1}\ar @{<-}[dl]_{w_1} \ar @{<-} [dr]^{w_2}&&{4}\ar @{<-} [dl]_{w_3}\ar[dr]^{w_4}\\
{2} &&{3}&&{a}
}\hskip .5cm$$
In this case the vertex $a$ can be either $2$ or $3$ or $4$ or $5$.
\item[$(ii)$]
$$\hskip .5cm 
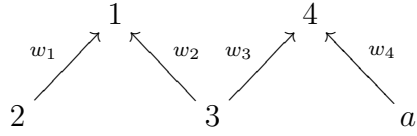
\xymatrix{
&{1}\ar @{<-}[dl]_{w_1} \ar @{<-} [dr]^{w_2}&&{4}\ar @{<-} [dl]_{w_3}\ar @{<-} [dr]^{w_4}\\
{2} &&{3}&&{a}
}\hskip .5cm$$
In this case the vertex $a$ can be either $1$ or $4$ or $5$.
\item[$(iii)$]
$$\hskip .5cm 
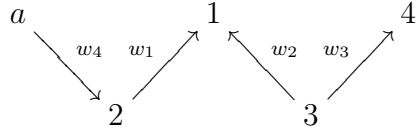
\xymatrix{
{a}\ar[dr]^{w_4}&&{1}\ar @{<-}[dl]_{w_1} \ar @{<-} [dr]^{w_2}&&{4}\ar @{<-} [dl]_{w_3}\\
&{2} &&{3}&
}\hskip .5cm$$
In this case the vertex $a$ can be either $1$ or $2$ or $5$.
\item[$(iv)$]
$$\hskip .5cm 
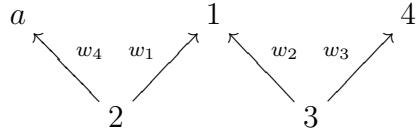
\xymatrix{
{a}\ar@{<-} [dr]^{w_4}&&{1}\ar @{<-}[dl]_{w_1} \ar @{<-} [dr]^{w_2}&&{4}\ar @{<-} [dl]_{w_3}\\
&{2} &&{3}&
}\hskip .5cm$$
In this case the vertex $a$ can be either $2$ or $3$ or $5$.
\end{itemize}
\item Fifth case: The walk $W^{'}$ is of the form
   \begin{center}
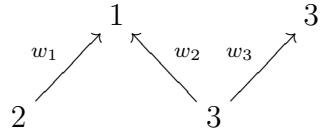
$$\hskip .5cm \xymatrix{
&{1}\ar @{<-}[dl]_{w_1} \ar @{<-} [dr]^{w_2}&&{3}\ar@{<-} [dl]_{w_3}\\
{2} &&{3}
}\hskip .5cm$$
\end{center}
In this case $W$ has a subwalk of one of the following forms:
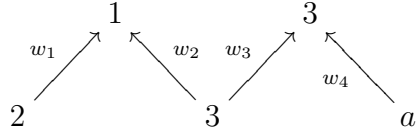
\begin{itemize}
\item[$(i)$]
$$\hskip .5cm \xymatrix{
&{1}\ar @{<-}[dl]_{w_1} \ar @{<-} [dr]^{w_2}&&{3}\ar@{<-} [dl]_{w_3}\ar@{<-} [dr]_{w_4}\\
{2} &&{3}&&{a}
}\hskip .5cm$$
In this case the vertex $a$ can be either $1$ or $4$.
\item[$(ii)$]
$$\hskip .5cm 
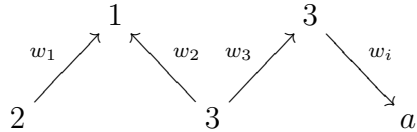
\xymatrix{
&{1}\ar @{<-}[dl]_{w_1} \ar @{<-} [dr]^{w_2}&&{3}\ar@{<-} [dl]_{w_3}\ar [dr]^{w_i}\\
{2} &&{3}&&{a}
}\hskip .5cm$$
for $i=2,3$. If $i=2$, then $a=1$ and if $i=3$, then $a=3$.
\item[$(iii)$]
$$\hskip .5cm 
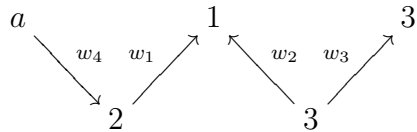
\xymatrix{
{a}\ar[dr]^{w_4}&&{1}\ar @{<-}[dl]_{w_1} \ar @{<-} [dr]^{w_2}&&{3}\ar@{<-} [dl]_{w_3}\\
&{2} &&{3}
}\hskip .5cm$$
In this case the vertex $a$ can be either $1$ or $2$  or $4$.
\item[$(iv)$]
$$\hskip .5cm 
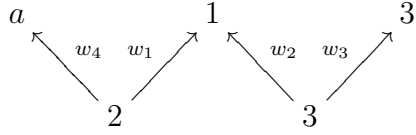
\xymatrix{
{a}\ar@{<-}[dr]^{w_4}&&{1}\ar @{<-}[dl]_{w_1} \ar @{<-} [dr]^{w_2}&&{3}\ar@{<-} [dl]_{w_3}\\
&{2} &&{3}
}\hskip .5cm$$
In this case the vertex $a$ can be either $2$ or $4$.
\end{itemize}

\item Sixth case: The walk $W^{'}$ is of the form
$$\hskip .5cm 
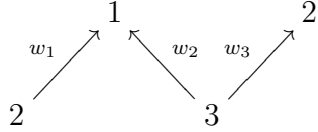
\xymatrix{
&{1}\ar @{<-}[dl]_{w_1} \ar @{<-} [dr]^{w_2}&&{2}\ar @{<-} [dl]_{w_3}\\
{2} &&{3}
}\hskip .5cm$$
In this case $W$ has a subwalk of one of the following forms:
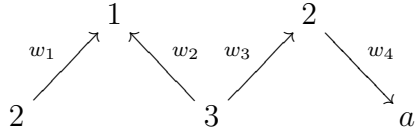
\begin{itemize}
\item[$(i)$]
$$\hskip .5cm \xymatrix{
&{1}\ar @{<-}[dl]_{w_1} \ar @{<-} [dr]^{w_2}&&{2}\ar @{<-} [dl]_{w_3}\ar[dr]^{w_4}\\
{2} &&{3}&&{a}
}\hskip .5cm$$
In this case the vertex $a$ can be either $3$ or $4$.
\item[$(ii)$]
$$\hskip .5cm 
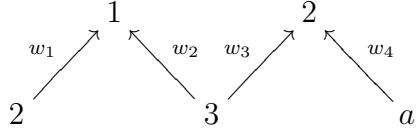
\xymatrix{
&{1}\ar @{<-}[dl]_{w_1} \ar @{<-} [dr]^{w_2}&&{2}\ar @{<-} [dl]_{w_3}\ar @{<-} [dr]^{w_4}\\
{2} &&{3}&&{a}
}\hskip .5cm$$
In this case the vertex $a$ can be either $1$ or $4$.
 \end{itemize}

\item Seventh case: The walk $W^{'}$ is of the form
$$\hskip .5cm 
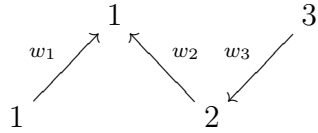
\xymatrix{
&{1}\ar @{<-}[dl]_{w_1} \ar @{<-} [dr]^{w_2}&&{3}\ar[dl]_{w_3}\\
{1} &&{2}
}\hskip .5cm$$
In this case $W$ has a subwalk of one of the following forms:
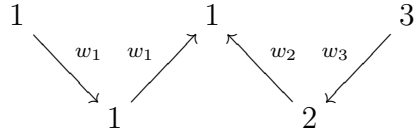
\begin{itemize}
\item[$(i)$]
$$\hskip .5cm \xymatrix{
{1}\ar[dr]^{w_1}&&{1}\ar @{<-}[dl]_{w_1} \ar @{<-} [dr]^{w_2}&&{3}\ar[dl]_{w_3}\\
&{1} &&{2}
}\hskip .5cm$$
\item[$(ii)$]
$$\hskip .5cm 
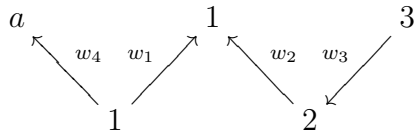
\xymatrix{
{a}\ar@{<-}[dr]^{w_4}&&{1}\ar @{<-}[dl]_{w_1} \ar @{<-} [dr]^{w_2}&&{3}\ar[dl]_{w_3}\\
&{1} &&{2}
}\hskip .5cm$$
In this case the vertex $a$ can be either $2$ or $3$ or $4$.
\item[$(iii)$]
$$\hskip .5cm 
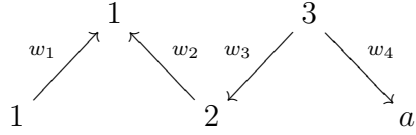
\xymatrix{
&{1}\ar @{<-}[dl]_{w_1} \ar @{<-} [dr]^{w_2}&&{3}\ar[dl]_{w_3}\ar[dr]^{w_4}\\
{1} &&{2}&&{a}
}\hskip .5cm$$
In this case the vertex $a$ can be either $3$ or $4$.
\item[$(iv)$]
$$\hskip .5cm 
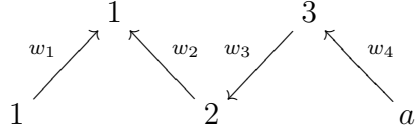
\xymatrix{
&{1}\ar @{<-}[dl]_{w_1} \ar @{<-} [dr]^{w_2}&&{3}\ar[dl]_{w_3}\ar@{<-}[dr]^{w_4}\\
{1} &&{2}&&{a}
}\hskip .5cm$$
In this case the vertex $a$ can be either $2$ or $3$ or $4$.
\end{itemize}

\item Eighth case: The walk $W^{'}$ is of the form
$$\hskip .5cm 
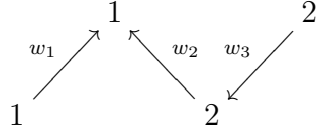
\xymatrix{
&{1}\ar @{<-}[dl]_{w_1} \ar @{<-} [dr]^{w_2}&&{2}\ar[dl]_{w_3}\\
{1} &&{2}
}\hskip .5cm$$
In this case $W$ has a subwalk of one of the following forms:
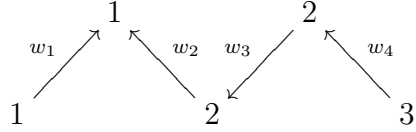
\begin{itemize}
\item[$(i)$]
$$\hskip .5cm \xymatrix{
&{1}\ar @{<-}[dl]_{w_1} \ar @{<-} [dr]^{w_2}&&{2}\ar[dl]_{w_3}\ar@{<-}[dr]^{w_4}\\
{1} &&{2}&&{3}
}\hskip .5cm$$

\item[$(ii)$]
$$\hskip .5cm 
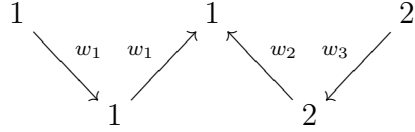
\xymatrix{
{1}\ar[dr]^{w_1}&&{1}\ar @{<-}[dl]_{w_1} \ar @{<-} [dr]^{w_2}&&{2}\ar[dl]_{w_3}\\
&{1} &&{2}
}\hskip .5cm$$
\item[$(iii)$]
$$\hskip .5cm 
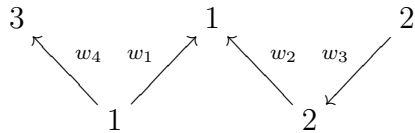
\xymatrix{
{3}\ar@{<-}[dr]^{w_4}&&{1}\ar @{<-}[dl]_{w_1} \ar @{<-} [dr]^{w_2}&&{2}\ar[dl]_{w_3}\\
&{1} &&{2}
}\hskip .5cm$$

\end{itemize}
\item  Ninth case: The walk $W^{'}$ is of the form
$$\hskip .5cm 
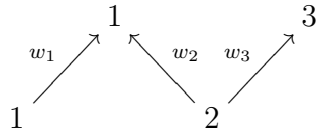
\xymatrix{
&{1}\ar @{<-}[dl]_{w_1} \ar @{<-} [dr]^{w_2}&&{3}\ar@{<-}[dl]_{w_3}\\
{1} &&{2}
}\hskip .5cm$$
In this case $W$ has a subwalk of one of the following  forms:
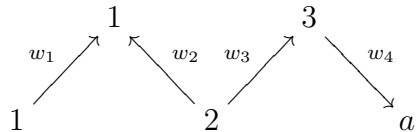
\begin{itemize}
\item[$(i)$]
$$\hskip .5cm \xymatrix{
&{1}\ar @{<-}[dl]_{w_1} \ar @{<-} [dr]^{w_2}&&{3}\ar@{<-}[dl]_{w_3}\ar[dr]^{w_4}\\
{1} &&{2}&&{a}
}\hskip .5cm$$
In this case the vertex $a$ can be either $2$ or $3$ or $4$.
\item[$(ii)$]
$$\hskip .5cm 
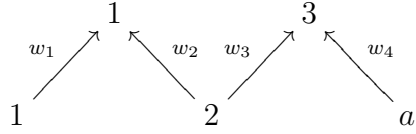
\xymatrix{
&{1}\ar @{<-}[dl]_{w_1} \ar @{<-} [dr]^{w_2}&&{3}\ar@{<-}[dl]_{w_3}\ar@{<-}[dr]^{w_4}\\
{1} &&{2}&&{a}
}\hskip .5cm$$
In this case the vertex $a$ can be either $3$ or $4$.
\item[$(iii)$]
$$\hskip .5cm 
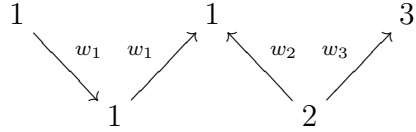
\xymatrix{
{1}\ar[dr]^{w_1}&&{1}\ar @{<-}[dl]_{w_1} \ar @{<-} [dr]^{w_2}&&{3}\ar@{<-}[dl]_{w_3}\\
&{1} &&{2}
}\hskip .5cm$$
\item[$(iv)$]
$$\hskip .5cm 
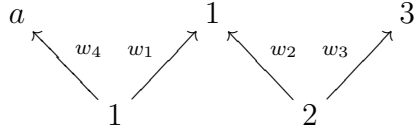
\xymatrix{
{a}\ar@{<-}[dr]^{w_4}&&{1}\ar @{<-}[dl]_{w_1} \ar @{<-} [dr]^{w_2}&&{3}\ar@{<-}[dl]_{w_3}\\
&{1} &&{2}
}\hskip .5cm$$
In this case the vertex $a$ can be either $2$ or $4$.
\end{itemize}
\item Tenth case: The walk $W^{'}$ is of the form
$$\hskip .5cm 
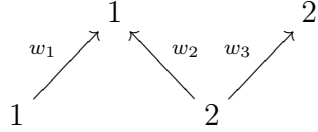
\xymatrix{
&{1}\ar @{<-}[dl]_{w_1} \ar @{<-} [dr]^{w_2}&&{2}\ar@{<-}[dl]_{w_3}\\
{1} &&{2}
}\hskip .5cm$$
In this case $W$ has a subwalk of one of the following forms:
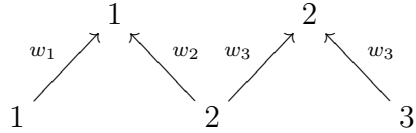
\begin{itemize}
\item[$(i)$]
$$\hskip .5cm \xymatrix{
&{1}\ar @{<-}[dl]_{w_1} \ar @{<-} [dr]^{w_2}&&{2}\ar@{<-}[dl]_{w_3}\ar@{<-}[dr]^{w_3}\\
{1} &&{2}&&{3}
}\hskip .5cm$$

\item[$(ii)$]
$$\hskip .5cm 
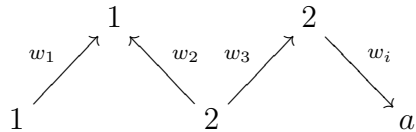
\xymatrix{
&{1}\ar @{<-}[dl]_{w_1} \ar @{<-} [dr]^{w_2}&&{2}\ar@{<-}[dl]_{w_3}\ar[dr]^{w_i}\\
{1} &&{2}&&{a}
}\hskip .5cm$$
If $i=2$, then $a=1$ and if $i=3$, then $a=2$ .
\item[$(iii)$]
$$\hskip .5cm 
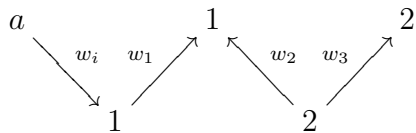
\xymatrix{
{a}\ar[dr]^{w_i}&&{1}\ar @{<-}[dl]_{w_1} \ar @{<-} [dr]^{w_2}&&{2}\ar@{<-}[dl]_{w_3}\\
&{1} &&{2}}\hskip .5cm$$
If $i=2$, then $a=2$ and if $i=1$, then $a=1$.
\item[$(iv)$]
$$\hskip .5cm 
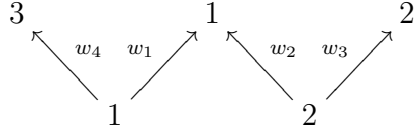
\xymatrix{
{3}\ar@{<-}[dr]^{w_4}&&{1}\ar @{<-}[dl]_{w_1} \ar @{<-} [dr]^{w_2}&&{2}\ar@{<-}[dl]_{w_3}\\
&{1} &&{2}
}\hskip .5cm$$

\end{itemize}
\item Eleventh case: The walk $W^{'}$ is of the form
$$\hskip .5cm 
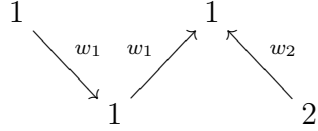
\xymatrix{
{1}\ar[dr]^{w_1}&&{1}\ar @{<-}[dl]_{w_1} \ar @{<-} [dr]^{w_2}\\
&{1} &&{2}
}\hskip .5cm$$
In this case $W$ has a subwalk of one of the following forms:
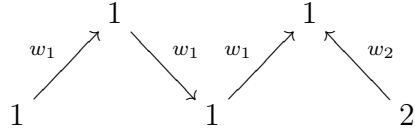
\begin{itemize}
\item[$(i)$]
$$\hskip .5cm \xymatrix{
&{1}\ar@{<-}[dl]_{w_1}\ar[dr]^{w_1}&&{1}\ar @{<-}[dl]_{w_1} \ar @{<-} [dr]^{w_2}\\
{1}&&{1} &&{2}
}\hskip .5cm$$
\item[$(ii)$]
$$\hskip .5cm 
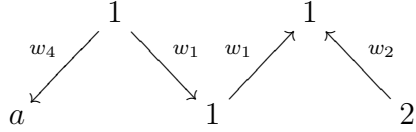
\xymatrix{
&{1}\ar[dl]_{w_4}\ar[dr]^{w_1}&&{1}\ar @{<-}[dl]_{w_1} \ar @{<-} [dr]^{w_2}\\
{a}&&{1} &&{2}
}\hskip .5cm$$
In this case the vertex $a$ can be either $2$ or $3$.
\item[$(iii)$]
$$\hskip .5cm 
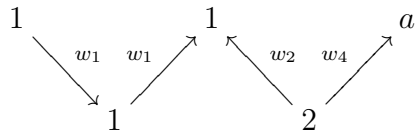
\xymatrix{
{1}\ar[dr]^{w_1}&&{1}\ar @{<-}[dl]_{w_1} \ar @{<-} [dr]^{w_2}&&{a}\ar@{<-}[dl]_{w_4}\\
&{1} &&{2}
}\hskip .5cm$$
In this case the vertex $a$ can be either $2$ or $3$.
\item[$(iv)$]
$$\hskip .5cm 
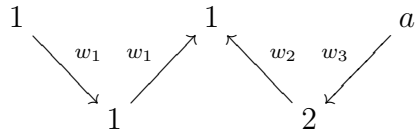
\xymatrix{
{1}\ar[dr]^{w_1}&&{1}\ar @{<-}[dl]_{w_1} \ar @{<-} [dr]^{w_2}&&{a}\ar[dl]_{w_3}\\
&{1} &&{2}
}\hskip .5cm$$
In this case the vertex $a$ can be either $2$ or $3$.
\end{itemize}
 \item Twelfth case: The walk $W^{'}$ is of the form
$$\hskip .5cm 
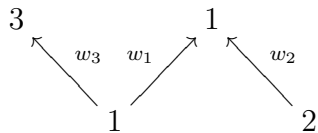
\xymatrix{
{3}\ar@{<-}[dr]^{w_3}&&{1}\ar @{<-}[dl]_{w_1} \ar @{<-} [dr]^{w_2}\\
&{1} &&{2}
}\hskip .5cm$$
In this case $W$ has a subwalk of one of the following forms:
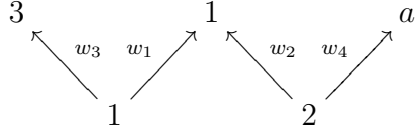
\begin{itemize}
\item[$(i)$]
$$\hskip .5cm \xymatrix{
{3}\ar@{<-}[dr]^{w_3}&&{1}\ar @{<-}[dl]_{w_1} \ar @{<-} [dr]^{w_2}&&{a}\ar@{<-}[dl]_{w_4}\\
&{1} &&{2}
}\hskip .5cm$$
In this case the vertex $a$ can be either $2$ or $4$.
\item[$(ii)$]
$$\hskip .5cm 
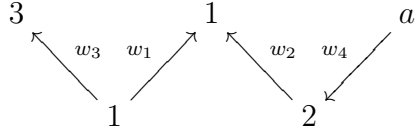
\xymatrix{
{3}\ar@{<-}[dr]^{w_3}&&{1}\ar @{<-}[dl]_{w_1} \ar @{<-} [dr]^{w_2}&&{a}\ar[dl]_{w_4}\\
&{1} &&{2}
}\hskip .5cm$$
In this case the vertex $a$ can be either $2$ or $3$ or $4$.
\item[$(iii)$]
$$\hskip .5cm 
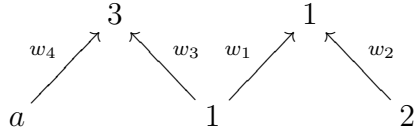
\xymatrix{
&{3}\ar@{<-}[dl]_{w_4}\ar@{<-}[dr]^{w_3}&&{1}\ar @{<-}[dl]_{w_1} \ar @{<-} [dr]^{w_2}\\
{a}&&{1} &&{2}
}\hskip .5cm$$
In this case the vertex $a$ can be either $ 3 $ or $ 4 $.
\item[$(iv)$]
$$\hskip .5cm 
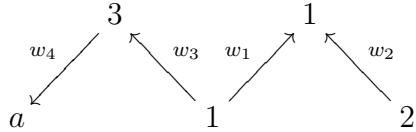
\xymatrix{
&{3}\ar[dl]_{w_4}\ar@{<-}[dr]^{w_3}&&{1}\ar @{<-}[dl]_{w_1} \ar @{<-} [dr]^{w_2}\\
{a}&&{1} &&{2}
}\hskip .5cm$$
In this case the vertex $a$ can be either  $2$ or $3$ or $4$.
\end{itemize}
 \item Thirteenth case: The walk $W^{'}$ is of the form
$$\hskip .5cm 
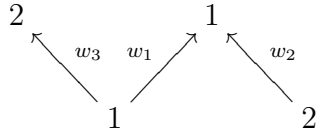
\xymatrix{
{2}\ar@{<-}[dr]^{w_3}&&{1}\ar @{<-}[dl]_{w_1} \ar @{<-} [dr]^{w_2}\\
&{1} &&{2}
}\hskip .5cm$$
In this case $W$ has a subwalk of one of the following forms:
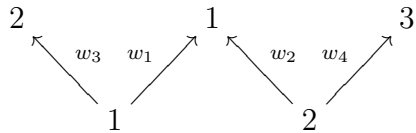
\begin{itemize}
\item[$(i)$]
$$\hskip .5cm \xymatrix{
{2}\ar@{<-}[dr]^{w_3}&&{1}\ar @{<-}[dl]_{w_1} \ar @{<-} [dr]^{w_2}&&{3}\ar@{<-}[dl]_{w_4}\\
&{1} &&{2}
}\hskip .5cm$$

\item[$(ii)$]
$$\hskip .5cm 
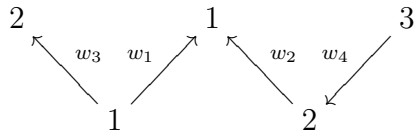
\xymatrix{
{2}\ar@{<-}[dr]^{w_3}&&{1}\ar @{<-}[dl]_{w_1} \ar @{<-} [dr]^{w_2}&&{3}\ar[dl]_{w_4}\\
&{1} &&{2}
}\hskip .5cm$$

\item[$(iii)$]
$$\hskip .5cm 
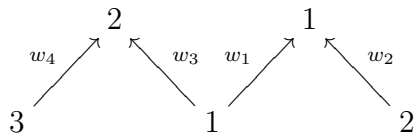
\xymatrix{
&{2}\ar@{<-}[dl]_{w_4}\ar@{<-}[dr]^{w_3}&&{1}\ar @{<-}[dl]_{w_1} \ar @{<-} [dr]^{w_2}\\
{3}&&{1} &&{2}
}\hskip .5cm$$

\item[$(iv)$]
$$\hskip .5cm 
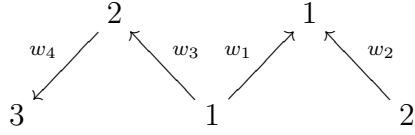
\xymatrix{
&{2}\ar[dl]_{w_4}\ar@{<-}[dr]^{w_3}&&{1}\ar @{<-}[dl]_{w_1} \ar @{<-} [dr]^{w_2}\\
{3}&&{1} &&{2}
}\hskip .5cm$$

\end{itemize}
 \end{itemize}
 In all the above cases, there exists an indecomposable $5$-factor serial right $\Lambda$-module which gives a contradiction.\\
 Assume on the contrary that the condition $(ii)$ does not hold. Then there exists a walk $W$ of length greater than or equal to $5$ which has a subwalk $W^{'}$ of length $4$ that $W^{'}$ has a subwalk of the form $w_{1}^{-1}w_{2}^{+1}$. By the condition $(i)$, $W$ has not a subwalk $W^{''}$ of length greater than or equal to $4$ that $W''$ has a subwalk of the form $w_{1}^{+1}w_{2}^{-1}$. Since $\Lambda$ is an algebra of finite type, the walk $W^{'}$ is one of the following forms.
 \begin{itemize}
 \item First case: The walk $W^{'}$ is of the form
 $$\hskip .5cm 
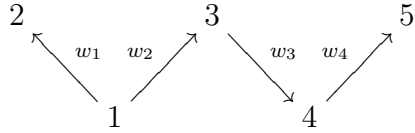
\xymatrix{
{2}\ar@{<-}[dr]^{w_1}&&{3}\ar @{<-}[dl]_{w_2} \ar  [dr]^{w_3}&&{5}\ar@{<-}[dl]_{w_4}\\
&{1} &&{4}
}\hskip .5cm$$
In this case $W$ has a subwalk of one of the following forms:
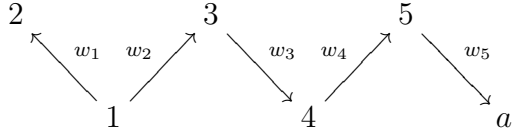
\begin{itemize}
\item[$(i)$]
 $$\hskip .5cm \xymatrix{
{2}\ar@{<-}[dr]^{w_1}&&{3}\ar @{<-}[dl]_{w_2} \ar  [dr]^{w_3}&&{5}\ar@{<-}[dl]_{w_4}\ar[dr]^{w_5}\\
&{1} &&{4}&&{a}
}\hskip .5cm$$
In this case the vertex $a$ can be either $6$ or $1$.
\item[$(ii)$]
 $$\hskip .5cm 
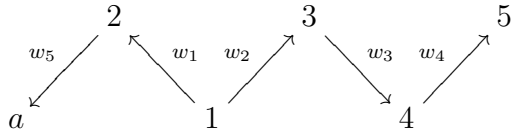
\xymatrix{
&{2}\ar@{<-}[dr]^{w_1}\ar[dl]_{w_5}&&{3}\ar @{<-}[dl]_{w_2} \ar  [dr]^{w_3}&&{5}\ar@{<-}[dl]_{w_4}\\
{a}&&{1} &&{4}
}\hskip .5cm$$
In this case the vertex $a$ can be either $6$ or $1$.
\end{itemize}
\item Second case: The walk $W^{'}$ is of the form
 $$\hskip .5cm 
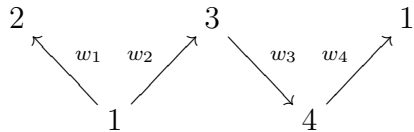
\xymatrix{
{2}\ar@{<-}[dr]^{w_1}&&{3}\ar @{<-}[dl]_{w_2} \ar  [dr]^{w_3}&&{1}\ar@{<-}[dl]_{w_4}\\
&{1} &&{4}
}\hskip .5cm$$
In this case $W$ has a subwalk of one of the following forms:
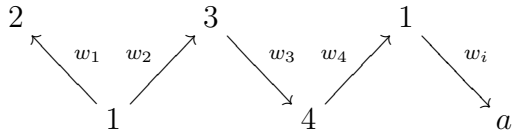
\begin{itemize}
\item[$(i)$]
 $$\hskip .5cm \xymatrix{
{2}\ar@{<-}[dr]^{w_1}&&{3}\ar @{<-}[dl]_{w_2} \ar  [dr]^{w_3}&&{1}\ar@{<-}[dl]_{w_4}\ar[dr]^{w_i}\\
&{1} &&{4}&&{a}
}\hskip .5cm$$
for $ i=1,2$. If $i=1$, then $a=2$ and if $i=2$, then $a=3$.

\item[$(ii)$]
 $$\hskip .5cm 
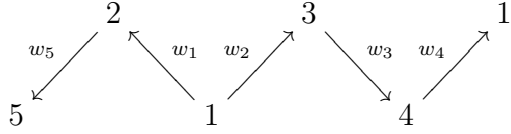
\xymatrix{
&{2}\ar@{<-}[dr]^{w_1} \ar[dl]_{w_5}&&{3}\ar @{<-}[dl]_{w_2} \ar  [dr]^{w_3}&&{1}\ar@{<-}[dl]_{w_4}\\
{5}&&{1} &&{4}
}\hskip .5cm$$
\end{itemize}

 \item Third case: The walk $W^{'}$ is of the form
 $$\hskip .5cm 
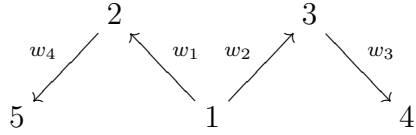
\xymatrix{
&{2}\ar@{<-}[dr]^{w_1} \ar[dl]_{w_4}&&{3}\ar @{<-}[dl]_{w_2}\ar[dr]^{w_3} \\
{5}&&{1} &&{4}
}
\hskip .5cm$$
In this case $W$ has a subwalk of the following form:
$$\hskip .5cm 
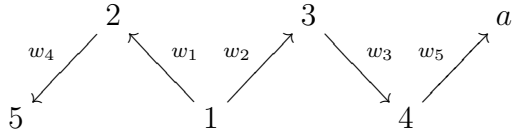
\xymatrix{
&{2}\ar@{<-}[dr]^{w_1} \ar[dl]_{w_4}&&{3}\ar @{<-}[dl]_{w_2}\ar[dr]^{w_3}&&{a}\ar@{<-}[dl]_{w_5} \\
{5}&&{1} &&{4}
}
\hskip .5cm$$
In this case the vertex $a$ can be either $6$ or $1$.
 \item Fourth case: The walk $W^{'}$ is of the form
 $$\hskip .5cm 
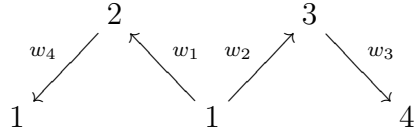
\xymatrix{
&{2}\ar@{<-}[dr]^{w_1} \ar[dl]_{w_4}&&{3}\ar @{<-}[dl]_{w_2}\ar[dr]^{w_3} \\
{1}&&{1} &&{4}
}
\hskip .5cm$$
In this case $W$ has a subwalk of one of the following forms:
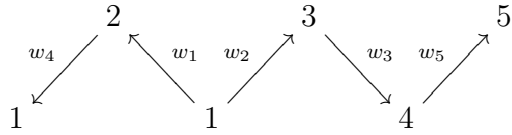
\begin{itemize}
\item[$(i)$]
$$\hskip .5cm \xymatrix{
&{2}\ar@{<-}[dr]^{w_1} \ar[dl]_{w_4}&&{3}\ar @{<-}[dl]_{w_2}\ar[dr]^{w_3}&&{5}\ar@{<-}[dl]_{w_5} \\
{1}&&{1} &&{4}
}
\hskip .5cm$$
\item[$(ii)$]
 $$\hskip .5cm 
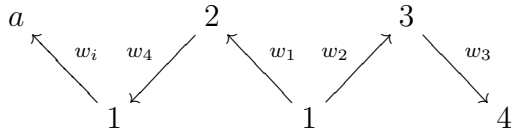
\xymatrix{
{a}\ar@{<-}[dr]^{w_i}&&{2}\ar@{<-}[dr]^{w_1} \ar[dl]_{w_4}&&{3}\ar @{<-}[dl]_{w_2}\ar[dr]^{w_3} \\
&{1}&&{1} &&{4}
}
\hskip .5cm$$
If $i=1$, then $a=2$ and if $i=2$, then $a=3$.
\end{itemize}

\item Fifth case: The walk $W^{'}$ is of the form
 $$\hskip .5cm 
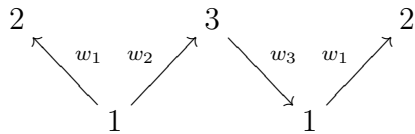
\xymatrix{
{2}\ar@{<-}[dr]^{w_1}&&{3}\ar @{<-}[dl]_{w_2} \ar  [dr]^{w_3}&&{2}\ar@{<-}[dl]_{w_1}\\
&{1} &&{1}
}\hskip .5cm$$
In this case $W$ has a subwalk of one of the following forms:
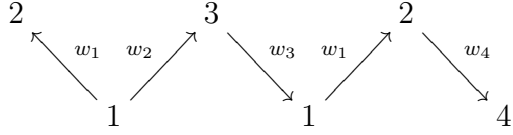
\begin{itemize}
\item[$(i)$]
 $$\hskip .5cm \xymatrix{
{2}\ar@{<-}[dr]^{w_1}&&{3}\ar @{<-}[dl]_{w_2} \ar  [dr]^{w_3}&&{2}\ar@{<-}[dl]_{w_1}\ar[dr]^{w_4}\\
&{1} &&{1}&&{4}
}\hskip .5cm$$
\item[$(ii)$]
  $$\hskip .5cm 
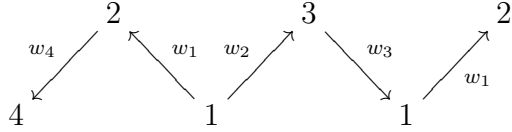
\xymatrix{
&{2}\ar@{<-}[dr]^{w_1}\ar[dl]_{w_4}&&{3}\ar @{<-}[dl]_{w_2} \ar  [dr]^{w_3}&&{2}\ar@{<-}[dl]^{w_1}\\
{4}&&{1} &&{1}
}\hskip .5cm$$
\end{itemize}

 \item Sixth case: The walk $W^{'}$ is of the form
 $$\hskip .5cm 
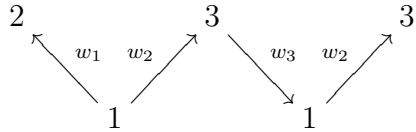
\xymatrix{
{2}\ar@{<-}[dr]^{w_1}&&{3}\ar @{<-}[dl]_{w_2} \ar  [dr]^{w_3}&&{3}\ar@{<-}[dl]_{w_2}\\
&{1} &&{1}
}\hskip .5cm$$
In this case $W$ has a subwalk of one of the following forms:
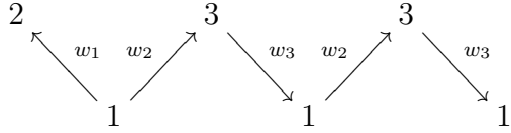
\begin{itemize}
\item[$(i)$]
 $$\hskip .5cm \xymatrix{
{2}\ar@{<-}[dr]^{w_1}&&{3}\ar @{<-}[dl]_{w_2} \ar  [dr]^{w_3}&&{3}\ar@{<-}[dl]_{w_2}\ar[dr]^{w_3}\\
&{1} &&{1}&&{1}
}\hskip .5cm$$
\item[$(ii)$]
  $$\hskip .5cm 
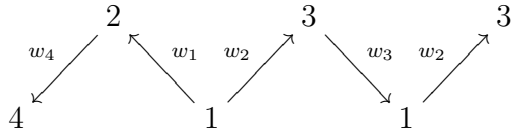
\xymatrix{
&{2}\ar@{<-}[dr]^{w_1}\ar[dl]_{w_4}&&{3}\ar @{<-}[dl]_{w_2} \ar  [dr]^{w_3}&&{3}\ar@{<-}[dl]_{w_2}\\
{4}&&{1} &&{1}
}\hskip .5cm$$
\end{itemize}
\item Seventh case: The walk $W^{'}$ is of the form
 $$\hskip .5cm 
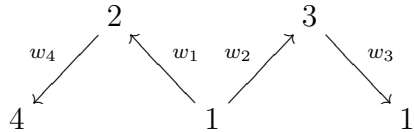
\xymatrix{
&{2}\ar@{<-}[dr]^{w_1}\ar[dl]_{w_4}&&{3}\ar @{<-}[dl]_{w_2} \ar  [dr]^{w_3}\\
{4}&&{1} &&{1}
}\hskip .5cm$$
In this case $W$ has a subwalk of one of the following forms:
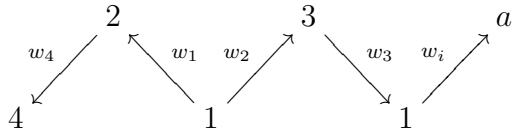
\begin{itemize}
\item[$(i)$]
 $$\hskip .5cm \xymatrix{
&{2}\ar@{<-}[dr]^{w_1}\ar[dl]_{w_4}&&{3}\ar @{<-}[dl]_{w_2} \ar  [dr]^{w_3}&&{a}\ar@{<-}[dl]_{w_i}\\
{4}&&{1} &&{1}
}\hskip .5cm$$
If $i=1$, then $a=2$ and if $i=2$, then $a=3$.
\item[$(ii)$]
  $$\hskip .5cm 
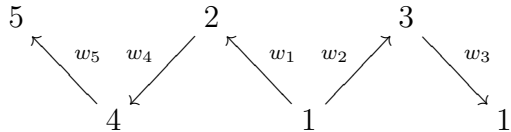
\xymatrix{
{5}\ar@{<-}[dr]^{w_5}&&{2}\ar@{<-}[dr]^{w_1}\ar[dl]_{w_4}&&{3}\ar @{<-}[dl]_{w_2} \ar  [dr]^{w_3}\\
&{4}&&{1}&&{1}
}\hskip .5cm$$
\end{itemize}
\item Eighth case: The walk $W^{'}$ is of the form
 $$\hskip .5cm 
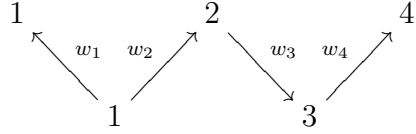
\xymatrix{
{1}\ar@{<-}[dr]^{w_1}&&{2}\ar @{<-}[dl]_{w_2} \ar  [dr]^{w_3}&&{4}\ar@{<-}[dl]_{w_4} \\
&{1} &&{3}
}\hskip .5cm$$
In this case $W$ has a subwalk of one of the following forms:
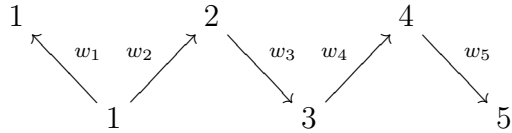
\begin{itemize}
\item[$(i)$]
$$\hskip .5cm \xymatrix{
{1}\ar@{<-}[dr]^{w_1}&&{2}\ar @{<-}[dl]_{w_2} \ar  [dr]^{w_3}&&{4}\ar @{<-}[dl]_{w_4}\ar[dr]^{w_5}\\
&{1} &&{3}&&{5}
}\hskip .5cm$$
\item[$(ii)$]
 $$\hskip .5cm 
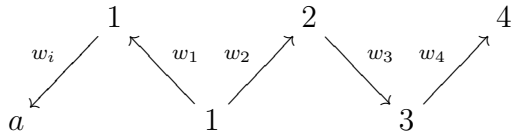
\xymatrix{
&{1}\ar@{<-}[dr]^{w_1}\ar[dl]_{w_i}&&{2}\ar @{<-}[dl]_{w_2} \ar  [dr]^{w_3}&&{4}\ar@{<-}[dl]_{w_4} \\
{a}&&{1} &&{3}
}\hskip .5cm$$
If $i=1$, then $a=1$ and if $i=2$, then $a=2$.
\end{itemize}
\item Ninth case: The walk $W^{'}$ is of the form
 $$\hskip .5cm 
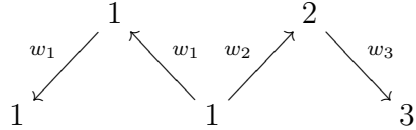
\xymatrix{
&{1}\ar@{<-}[dr]^{w_1}\ar[dl]_{w_1}&&{2}\ar @{<-}[dl]_{w_2} \ar  [dr]^{w_3}\\
{1}&&{1} &&{3}
}\hskip .5cm$$
In this case $W$ has a subwalk of one of  the following forms:
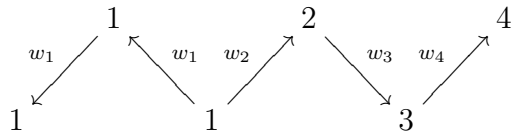
\begin{itemize}
\item[$(i)$]
$$\hskip .5cm \xymatrix{
&{1}\ar@{<-}[dr]^{w_1}\ar[dl]_{w_1}&&{2}\ar @{<-}[dl]_{w_2} \ar  [dr]^{w_3}&&{4}\ar @{<-}[dl]_{w_4}\\
{1}&&{1} &&{3}
}\hskip .5cm$$
\item[$(ii)$]
$$\hskip .5cm 
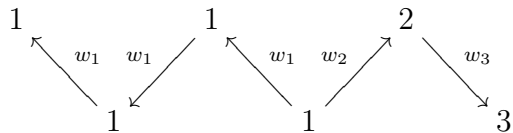
\xymatrix{
{1}\ar@{<-}[dr]^{w_1}&&{1}\ar@{<-}[dr]^{w_1}\ar[dl]_{w_1}&&{2}\ar @{<-}[dl]_{w_2} \ar  [dr]^{w_3}\\
&{1}&&{1} &&{3}
}\hskip .5cm$$

\end{itemize}
\item Tenth case: The walk $W^{'}$ is of the form
 $$\hskip .5cm 
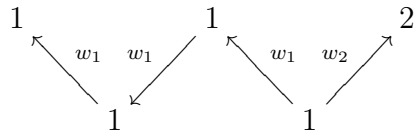
\xymatrix{
{1}\ar@{<-}[dr]^{w_1}&&{1}\ar@{<-}[dr]^{w_1}\ar[dl]_{w_1}&&{2}\ar @{<-}[dl]_{w_2} \\
&{1}&&{1}
}\hskip .5cm$$
In this case $W$ has a subwalk of one of the following forms:
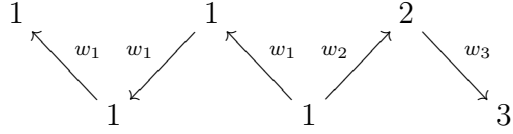
\begin{itemize}
\item[$(i)$]
  $$\hskip .5cm \xymatrix{
{1}\ar@{<-}[dr]^{w_1}&&{1}\ar@{<-}[dr]^{w_1}\ar[dl]_{w_1}&&{2}\ar @{<-}[dl]_{w_2}\ar[dr]^{w_3} \\
&{1}&&{1} &&{3}
}\hskip .5cm$$

\item[$(ii)$]
 $$\hskip .5cm 
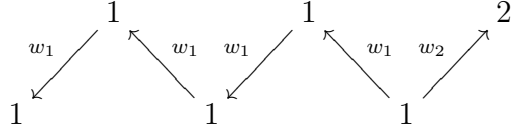
\xymatrix{
&{1}\ar@{<-}[dr]^{w_1}\ar[dl]_{w_1}&&{1}\ar@{<-}[dr]^{w_1}\ar[dl]_{w_1}&&{2}\ar @{<-}[dl]_{w_2} \\
{1}&&{1}&&{1}
}\hskip .5cm$$

 \end{itemize}
 \item Eleventh case: The walk $W^{'}$ is of the form
 $$\hskip .5cm 
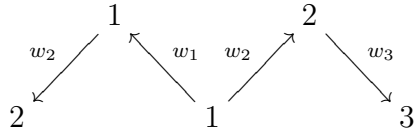
\xymatrix{
&{1}\ar@{<-}[dr]^{w_1}\ar[dl]_{w_2}&&{2}\ar @{<-}[dl]_{w_2} \ar  [dr]^{w_3}\\
{2}&&{1} &&{3}
}\hskip .5cm$$
In this case $W$ has a subwalk of one of the following forms:
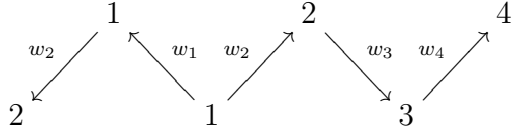
\begin{itemize}
\item[$(i)$]
$$\hskip .5cm \xymatrix{
&{1}\ar@{<-}[dr]^{w_1}\ar[dl]_{w_2}&&{2}\ar @{<-}[dl]_{w_2} \ar  [dr]^{w_3}&&{4}\ar @{<-}[dl]_{w_4}\\
{2}&&{1} &&{3}
}\hskip .5cm$$
\item[$(ii)$]
$$\hskip .5cm 
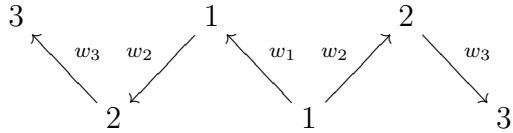
\xymatrix{
{3}\ar@{<-}[dr]^{w_3}&&{1}\ar@{<-}[dr]^{w_1}\ar[dl]_{w_2}&&{2}\ar @{<-}[dl]_{w_2} \ar  [dr]^{w_3}\\
&{2}&&{1} &&{3}
}\hskip .5cm$$

\end{itemize}
\item Twelfth case: The walk $W^{'}$ is of the form
 $$\hskip .5cm 
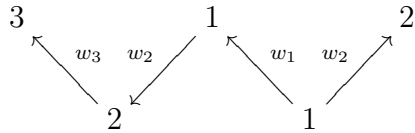
\xymatrix{
{3}\ar@{<-}[dr]^{w_3}&&{1}\ar@{<-}[dr]^{w_1}\ar[dl]_{w_2}&&{2}\ar @{<-}[dl]_{w_2} \\
&{2}&&{1}
}\hskip .5cm$$
In this case $W$ has a subwalk of one of the following forms:
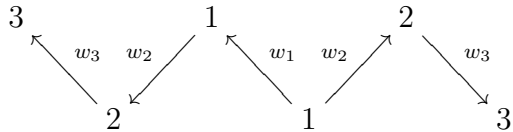
\begin{itemize}
\item[$(i)$]
  $$\hskip .5cm \xymatrix{
{3}\ar@{<-}[dr]^{w_3}&&{1}\ar@{<-}[dr]^{w_1}\ar[dl]_{w_2}&&{2}\ar @{<-}[dl]_{w_2}\ar[dr]^{w_3} \\
&{2}&&{1} &&{3}
}\hskip .5cm$$

\item[$(ii)$]
 $$\hskip .5cm 
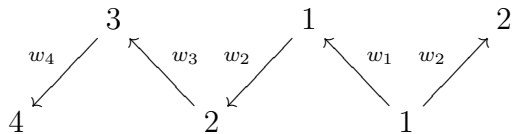
\xymatrix{
&{3}\ar@{<-}[dr]^{w_3}\ar[dl]_{w_4}&&{1}\ar@{<-}[dr]^{w_1}\ar[dl]_{w_2}&&{2}\ar @{<-}[dl]_{w_2} \\
{4}&&{2}&&{1}
}\hskip .5cm$$

 \end{itemize}
 \end{itemize}
 In all the above cases there exists an indecomposable $5$-factor serial right $\Lambda$-module which gives a contradiction.\\
 Now assume that the condition $(iii)$ does not hold. Then there exist paths $p=p_1...p_t$ and $q=q_1...q_s$ with the same target and the same source that $p_i$ and $q_j$ are arrows, $s\geq 2$, $t\geq 2$, $p-q\in I$ and $s+t>5$.
 $$\hskip .5cm 
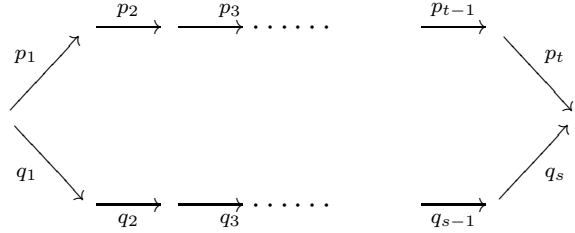
\xymatrix{
&{}\ar @{<-}[dl]_{p_{1}}\ar [r]^{p_{2}}&{}\ar[r]^{p_{3}}&\cdots\cdots&\ar[r]^{p_{t-1}}&\ar[dr]^{p_{t}}\\
{}\ar[dr]_{q_1}&&&&&&{}\ar @{<-}[dl]^{q_s}\\
&{}\ar [r]_{q_{2}}&{}\ar[r]_{q_3}&\cdots\cdots&\ar[r]_{q_{s-1}}&}\hskip .5cm$$
  We have two cases. In the first case, we have $t\geq 2$ and $s\geq 4$. Then there exists a string $w=q_{s-2}^{+1}q_{s-1}^{+1}q_{s}^{+1}p_{t}^{-1}$ of $\frac{KQ}{I}$, that $M(w)$ is a $5$-factor serial right $\Lambda$-module which gives a contradiction.
 In the second case, we have $t\geq 3$ and $s\geq 3$. Then there exists a string $w=p_{t-1}^{+1}p_{t}^{+1}q_{s}^{-1}q_{s-1}^{-1}$ of $\frac{KQ}{I}$, that $M(w)$ is a $5$-factor serial right $\Lambda$-module which gives a contradiction. Now assume that the condition $(iv)$ does not hold. By \cite[Theorem 5.13]{NS} and \cite[Theorem 3.3]{NS1} in this case $\Lambda$ is a $t$-Nakayama algebra for some $t\leq 3$, which gives a contradiction.\\
   Conversely, assume that $\Lambda$ is a special biserial algebra of finite type that $(Q, I)$ satisfies the conditions $(i)-(iv)$. By \cite{BR} and \cite{WW}, every indecomposable right $\Lambda$-module is either a string $M(w)$, $w\in \widetilde{\mathcal{S}}(\Lambda)$ or a band module $M(v, m, \varphi)$, $v\in \widetilde{\mathcal{B}}(\Lambda)$, $m\geq 1$ and $\varphi\in Aut(K^m)$ or non-uniserial projective-injective. Since $\Lambda$ is representation-finite, $\mathcal{B}(\Lambda)=\varnothing$. For any $w\in \widetilde{\mathcal{S}}(\Lambda)$, $w$ is either $w_1^{+1}...w_n^{+1}$ or $w_1^{-1}w_2^{+1}$ or $ w_1^{+1}w_2^{-1} $ or $ w_1^{-1}w_{2} ^{-1}w_3^{+1} $ or $ w_1^{-1}w_2^{-1}w_3^{+1}w_4^{+1} $ or $ w_1^{-1}w_2^{+1}w_3^{+1}w_4^{+1} $ or $w_1^{+1}w_{2} ^{+1}w_3^{-1}$ or $ w_1^{-1}w_{2} ^{+1}w_3^{-1} $. If $w=w_1^{+1}...w_n^{+1}$, then $M(w)$ is uniserial. If $w=w_1^{-1}w_2^{+1}$, then $M(w)$ is $2$-factor serial. If  either $w=w_1^{+1}w_2^{-1}$ or $w=w_1^{-1}w_{2} ^{-1}w_3^{+1}$ , then $M(w)$ is $3$-factor serial. If either $ w_1^{-1}w_2^{-1}w_3^{+1}w_4^{+1} $  or  $ w_1^{-1}w_2^{+1}w_3^{+1}w_4^{+1} $ or $w_1^{+1}w_{2} ^{+1}w_3^{-1}$  or $ w_1^{-1}w_{2} ^{+1}w_3^{-1} $, then $M(w)$ is a $4$-factor serial right $\Lambda$-module. By the condition $\left( \mathrm{iv}\right)$, there exists at least one string module $M(w)$, where $w$ is either $ w_1^{-1}w_2^{-1}w_3^{+1}w_4^{+1} $  or  $ w_1^{-1}w_2^{+1}w_3^{+1}w_4^{+1} $ or $w_1^{+1}w_{2} ^{+1}w_3^{-1}$  or $ w_1^{-1}w_{2} ^{+1}w_3^{-1} $. Therefore $\Lambda$ is a right $4$-Nakayama and the result follows.

\end{proof}

\begin{remark}
\item[$(1)$] Let $\Lambda=KQ/I$ be a right $4$-Nakayama finite dimensional $K$-algebra that the condition $(iv)(c)$ of the theorem \ref{T5} holds. Then there exists a non-uniserial projective-injective right $\Lambda$-module $M$, that $M$ is either $3$-factor serial or $4$-factor serial.
 \item[$(2)$] Let $\Lambda=\frac{KQ}{I}$ be a basic, connected and finite dimensional $K$-algebra. By \cite[Proposition 4.2]{NS1} $\Lambda$ is right $4$-Nakayama self-injective if and only if $Q$ is the following quiver with $s\geq 1$ and $m,n\geq 2$,
$$
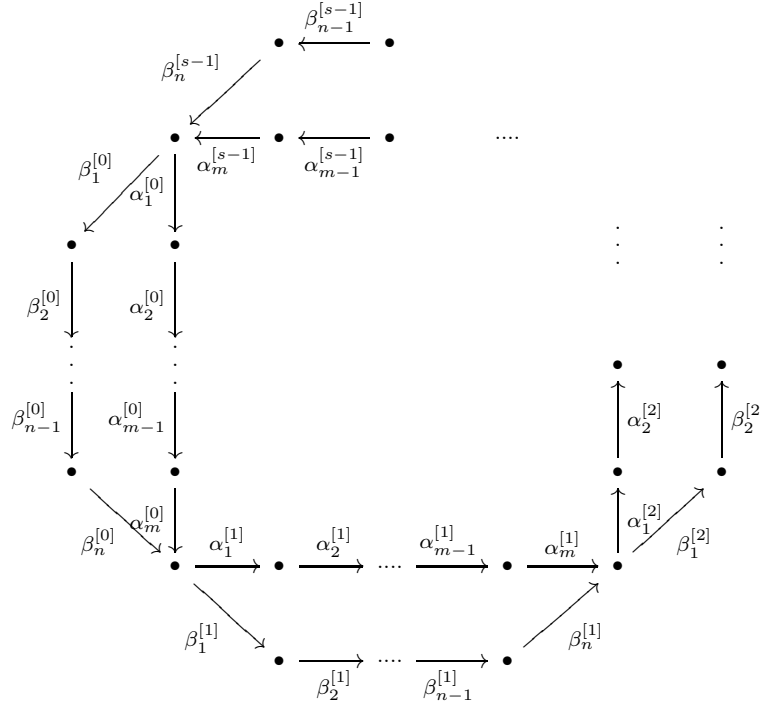
\xymatrix{
&& {\bsm\bullet \esm}\ar[dl]_{\beta_{n}^{[s-1]}}&{\bsm \bullet \esm}\ar[l]_{\beta_{n-1}^{[s-1]}}\\
&{\bsm\bullet\esm}\ar[dl]_{\beta_{1}^{[0]}}\ar[d]_{\alpha_{1}^{[0]}}&{\bsm\bullet \esm}\ar[l]^{\alpha_{m}^{[s-1]}} &{\bsm\bullet\esm}\ar[l]^{\alpha_{m-1}^{[s-1]}}&{\bsm....\esm} \\
{\bsm\bullet\esm}\ar[d]_{\beta_{2}^{[0]}}&{\bsm\bullet\esm}\ar[d]_{\alpha_{2}^{[0]}}&&&&{\bsm.\\.\\.\esm}&{\bsm.\\.\\.\esm} \\
{\bsm.\\.\\.\esm}\ar[d]_{\beta_{n-1}^{[0]}}&{\bsm.\\.\\.\esm}\ar[d]_{\alpha_{m-1}^{[0]}}&&&&{\bsm\bullet\esm}&{\bsm\bullet\esm}\\
{\bsm\bullet\esm}\ar[dr]_{\beta_{n}^{[0]}}&{\bsm\bullet\esm}\ar[d]_{\alpha_{m}^{[0]}}&&&&{\bsm\bullet\esm}\ar[u]_{\alpha_{2}^{[2]}}&{\bsm\bullet\esm}\ar[u]_{\beta_{2}^{[2]}} \\
&{\bsm\bullet\esm}\ar[r]^{\alpha_{1}^{[1]}}\ar[dr]_{\beta_{1}^{[1]}}&{\bsm\bullet\esm}\ar[r]^{\alpha_{2}^{[1]}}&{\bsm....\esm}\ar[r]^{\alpha_{m-1}^{[1]}}&{\bsm\bullet\esm}\ar[r]^{\alpha_{m}^{[1]}}&{\bsm\bullet\esm}\ar[u]_{\alpha_{1}^{[2]}} \ar[ur]_{\beta_{1}^{[2]}}\\
&&{\bsm \bullet \esm}\ar[r]_{\beta_{2}^{[1]}}&{\bsm....\esm}\ar[r]_{\beta_{n-1}^{[1]}}&{\bsm\bullet\esm}\ar[ur]_{\beta_{n}^{[1]}} \\}
$$
and $I$ is the ideal generated by the following relations:
 \begin{itemize}
 \item[$(i)$] $\alpha_{1}^{[i]}\cdots\alpha_{m}^{[i]}=\beta_{1}^{[i]}\cdots\beta_{n}^{[i]}$ for all $i\in \lbrace 0,\cdots,s-1\rbrace$;
  \item[$(ii)$] $\beta_n^{[i]}\alpha_{1}^{[i+1]}=0$, $\alpha_m^{[i]}\beta_{1}^{[i+1]}=0$ for all $i\in \lbrace 0,\cdots,s-2\rbrace$,  $\beta_{n}^{[s-1]} \alpha_{1}^{[0]}=0$ and $\alpha_m^{[s-1]}\beta_{1}^{[0]}=0$;
 \item[$(iii)$]
 \begin{itemize}
\item[$(a)$] Paths of the form $\alpha_i^{[j]}...\alpha_h^{[f]}$ of  length $m+1$ are equal to $0$.
\item[$(b)$] Paths of the form $\beta_i^{[j]}...\beta_h^{[f]}$ of length $n+1$ are equal to $0$.
\end{itemize}
\end{itemize}
 that $m+n=5$.
\end{remark}

\section*{acknowledgements} The research of the first
author was in part supported by a grant from IPM (No. 96170419).



\begin{thebibliography}{10}
\bibitem{A2} \textsc{H. Asashiba}, On algebras of second local type II, \emph{ Osaka J. Math.} \textbf{21} (1984), 343-364.

\bibitem{A1} \textsc{H. Asashiba}, On algebras of second local type III, \emph{ Osaka J. Math.} \textbf{24} (1987), 107-122.


\bibitem{ARS} \textsc{M. Auslander, I. Reiten, S. O. Smal{\o}}, \emph{Representation theory of Artin algebras}, Cambridge studies in
advanced mathematics \textbf{36}, Cambridge University Press, Cambridge, 1995.

\bibitem{BR} \textsc{M. C. R.  Butler, C. M. Ringel}, Auslander-Reiten sequences with few middle terms, with applications to string algebras, \emph{Comm. Algebra} \textbf{15} (1987), 145-179.

\bibitem{NS1} \textsc{A. Nasr-Isfahani, M. Shekari}, Representations of right 3-Nakayama algebras, arXiv:1805.04412.

\bibitem{NS} \textsc{A. Nasr-Isfahani, M. Shekari}, Right $n$-Nakayama algebras and their representations, arXiv:1710.01176.

\bibitem{SY} \textsc{A. Skowro´nski, K. Yamagata}, \emph{Frobenius algebras I: Basic representation
theory}, EMS Textbooks in Mathematics, European Mathematical Society (EMS),
Z¨urich, 2011.

\bibitem{T2} \textsc{H. Tachikawa}, On algebras of which every indecomposable representation has an irreducible one as the top or the bottom Loewy constituent, \emph{Math. Z.} \textbf{75} (1961), 215-227.

\bibitem{T1} \textsc{H. Tachikawa}, On rings for which every indecomposable right module has a unique maximal submodule, \emph{Math. Z.} \textbf{71} (1959), 200-222.


\bibitem{WW} \textsc{B. Wald, J. Waschbusch}, Tame biserial algebras,  \emph{J. Algebra} \textbf{15} (1985), no. 2, 480-500.

\bibitem{Z} \textsc{A. Zimmermann}, \emph{Representation theory, A homological algebra point of view},
Algebra and Applications, vol. 19, Springer, Cham, 2014.
\end{thebibliography}
\end{document}